\documentclass[10pt,a4paper]{amsart}

\usepackage{amsmath}
\usepackage{amsfonts}
\usepackage{amssymb}
\usepackage[colorlinks=true, pdfstartview=FitV, linkcolor=blue, citecolor=blue, urlcolor=blue,pagebackref=false]{hyperref}

\numberwithin{equation}{section}
\newtheorem{theorem}{Theorem}[section]
\newtheorem{corollary}[theorem]{Corollary}
\newtheorem{proposition}[theorem]{Proposition}
\newtheorem{lemma}[theorem]{Lemma}
\theoremstyle{definition}

\newtheorem{remark}[theorem]{Remark}
\newtheorem*{acknowledgments}{Acknowledgments}

\newcommand{\R}{\mathbb{R}} 
\newcommand{\C}{\mathbb{C}} 
\newcommand{\N}{\mathbb{N}} 
\newcommand{\dd}{\mathrm{d}} 
\newcommand{\tr}{\mathrm{tr}} 

\newcommand{\norm}[3]{\left\|#1\right\|^{#2}_{#3}} 
\newcommand{\duality}[2]{\langle #1 | #2 \rangle} 
\newcommand{\Duality}[2]{\left\langle #1 \Big| #2 \right\rangle} 
\newcommand{\inner}[2]{\left\langle #1, #2 \right\rangle} 
\newcommand{\Binner}[2]{\Big\langle #1, #2 \Big\rangle} 

\DeclareMathOperator{\supp}{supp}

\usepackage[usenames]{color}
\usepackage{xcolor}

\newcommand{\sh}{\mathcal{S}} 

\author[]{Pedro Caro \and Ting Zhou}
\title[]{On global uniqueness for an IBVP for the time-harmonic Maxwell equations}
\date{October 28, 2012}
\address{Department of Mathematics and Statistics, Helsingin yliopisto / Helsingfors universitet / University of Helsinki, Finland}
\email{pedro.caro@helsinki.fi}
\address{Department of Mathematics, Massachusetts Institute of Technology, USA}
\email{tzhou@math.mit.edu}

\begin{document}

\begin{abstract} In this paper we prove uniqueness for an inverse boundary value problem (IBVP) arising in electrodynamics. We assume that the electromagnetic properties of the medium, namely the magnetic permeability, the electric permittivity and the conductivity, are described by continuously differentiable functions.
\end{abstract}

\maketitle

\tableofcontents
\setcounter{tocdepth}{1}

\section{Introduction}
Let $ \Omega $ be a bounded non-empty open subset of $ \R^3 $ with boundary denoted by $ \partial \Omega $. Consider functions $ \mu, \varepsilon, \sigma \in L^\infty (\Omega) $, representing magnetic permeability, electric permittivity and conductivity respectively, such that $ \mu(x) \geq \mu_0 $, $ \varepsilon(x) \geq \varepsilon_0 $, $ \sigma(x) \geq 0 $ almost everywhere in $ \Omega $ for positive constants $ \mu_0 $ and $ \varepsilon_0 $. At frequency $ \omega > 0 $, for each medium characterized by $(\mu,\varepsilon,\sigma)$, we have access to all available data of the boundary tangential components of electric and magnetic fields. More specifically, we have access to the \textit{Cauchy data set} $ C(\mu,\varepsilon,\sigma; \omega) $ consisting of all boundary graded forms $f^1+f^2\in TH^\delta (\partial \Omega; \Lambda^{1} \R^3) \oplus TH^d(\partial \Omega; \Lambda^{2} \R^3)$ (see Appendix \ref{apx:A} for the definitions of these spaces and results related to $l$-forms) such that there exists $ u^1 + u^2 \in H^d(\Omega; \Lambda^1 \R^3) \oplus H^\delta (\Omega; \Lambda^2 \R^3) $ satisfying
\begin{equation}
\delta u^2 + i \omega \varepsilon u^1 - d u^1 + i \omega \mu u^2 = \sigma u^1 \label{eq:t-a_maxwell}
\end{equation}
almost everywhere in $ \Omega $ and
\begin{equation}
\delta \tr\, u^2 + d \tr\, u^1 = f^1 + f^2 \label{eq:boundary_data}
\end{equation}
in the sense of $ TH^\delta (\partial \Omega; \Lambda^{1} \R^3) \oplus TH^d(\partial \Omega; \Lambda^{2} \R^3) $. Here $u^1$ is the 1-form representation of the electric field and $u^2$ is the 2-form representation of the magnetic field. It is worth to point out that the graded equations \eqref{eq:t-a_maxwell} and \eqref{eq:boundary_data} are equivalent to the following systems of time-harmonic Maxwell equations
\begin{equation*}
\left\{
\begin{aligned}
& \delta u^2 + i \omega \varepsilon u^1 = \sigma u^1 \\
& d u^1 - i \omega \mu u^2 = 0
\end{aligned}
\right.\qquad
\end{equation*}
almost every where in $ \Omega $ and
\begin{equation*}
\left\{
\begin{aligned}
& \delta \tr\, u^2 = f^1 \\
& d \tr\, u^1 = f^2
\end{aligned}
\right.
\end{equation*}
in the sense of the space $ TH^\delta (\partial \Omega; \Lambda^{1} \R^3) $ for the $ 1 $-form equation and in the sense of $ TH^d(\partial \Omega; \Lambda^{2} \R^3) $ for the $ 2 $-form equation. Throughout this paper, we follow for convenience the graded form notation rather than the $ l $-form system.

We are interested in the inverse boundary value problem (IBVP for short) of recovering $ \mu, \varepsilon, \sigma \in L^\infty (\Omega) $ from the knowledge of $ C(\mu,\varepsilon,\sigma; \omega) $. This problem is just a reformulation in differential forms of the usual IBVP for the time-harmonic Maxwell equations proposed in \cite{SIsCh}, where $ \partial \Omega $ was smooth enough, the electromagnetic fields $(\mathbf{E}, \mathbf H)$ satisfied
\begin{equation*}
\left\{
\begin{aligned}
& \nabla\times \mathbf{E}-i\omega\mu \mathbf{H}=0\\ 
& \nabla\times \mathbf{H}+i\omega(\varepsilon+i\sigma/\omega)\mathbf{E}=0
\end{aligned}
\right.
\end{equation*}
almost everywhere in $ \Omega $ and the Cauchy set $C(\mu,\varepsilon,\sigma;\omega)$ consisted of pairs $(\nu\times\mathbf{E}|_{\partial\Omega}, \nu\times\mathbf{H}|_{\partial\Omega})\in TH^{1/2}_{\mathrm{Div}}(\partial\Omega)\times TH^{1/2}_{\mathrm{Div}}(\partial\Omega)$ (see \cite{SIsCh} for precise definitions) with $\nu$ denoting the unit outer normal vector to $\partial\Omega$. The uniqueness question associated to this problem is as follows. Given  a frequency $\omega>0$ and two sets of parameters $ \{ \mu_j, \varepsilon_j, \sigma_j \} \subset L^\infty (\Omega) $ with $ j \in \{ 1, 2 \} $ such that $ \mu_j(x) \geq \mu_0 $, $ \varepsilon_j(x) \geq \varepsilon_0 $, $ \sigma_j(x) \geq 0 $ almost everywhere in $ \Omega $, does $ C(\mu_1,\varepsilon_1,\sigma_1; \omega) = C(\mu_2,\varepsilon_2,\sigma_2; \omega) $ imply $ \mu_1 = \mu_2 $, $ \varepsilon_1 = \varepsilon_2 $ and $ \sigma_1 = \sigma_2 $?

In this paper we provide the answer to this question in the case that $ \Omega $ is locally described by the graph of a Lipschitz function and $ \mu, \varepsilon, \sigma $ are continuously differentiable in $ \Omega $. It is stated as the following main theorem.
\begin{theorem}\label{th:main} \sl Let $ \Omega $ be a bounded non-empty open subset of $ \R^3 $. Assume that $ \partial \Omega $ is locally described by the graph of a Lipschitz function. Let $ \mu_j $, $ \varepsilon_j $ and $ \sigma_j $ with $ j \in \{ 1, 2 \} $ belong to $ C^1(\overline{\Omega}) $. At frequency $\omega>0$, suppose $ \partial^\alpha \mu_1 (x) = \partial^\alpha \mu_2 (x) $, $ \partial^\alpha \varepsilon_1 (x) = \partial^\alpha \varepsilon_2 (x) $ and $ \partial^\alpha \sigma_1 (x) = \partial^\alpha \sigma_2 (x) $ for $ \alpha \in \N^3 $ with $ |\alpha| \leq 1 $ and all $ x \in \partial \Omega $, then 
$$ C(\mu_1,\varepsilon_1,\sigma_1, \omega) = C(\mu_2,\varepsilon_2,\sigma_2, \omega)\;\; \Longrightarrow\;\; \mu_1 = \mu_2, \varepsilon_1 = \varepsilon_2 \mbox{ and } \sigma_1 = \sigma_2 .$$
\end{theorem}
A precise definition of the space denoted by $ C^1(\overline{\Omega}) $ is given at the beginning of Section \ref{sec:integral_formula}. Our result assumes the coefficients to be equal up to order one on the boundary. This is required to extend them identically outside the domain. As far as we know, the only available results about uniqueness on the boundary in this context are due to Joshi and McDowall, where $ \partial \Omega $ is assumed to be locally described by a smooth function and the Cauchy data sets are given by the graph of a bounded map (see \cite{Mc} and \cite{JoMc}). 

The IBVP considered in this paper was first proposed by Somersalo, Isaacson and Cheney in \cite{SIsCh}. In \cite{L} Lassas found a relation between this IBVP and the inverse conductivity problem proposed by Calder\'on in \cite{C}. In general terms, the latter problem can be seen as low-frequency limit of the former one. Calder\'on's problem in Electrical Impedance Tomography consists in reconstructing the conductivity of a domain by measuring electric voltages and currents on the boundary. The uniqueness question arising in this problem is whether the conductivity $\sigma$ ($ \sigma \in L^\infty(\Omega) $ and $ \sigma (x) \geq \sigma_0 > 0 $ for almost every $ x \in \Omega $), in a divergence type equation $\nabla\cdot(\sigma\nabla u)=0$ in $\Omega$, can be determined uniquely by the boundary Dirichlet-to-Neumann map $ \Lambda_\sigma : H^1(\Omega)/H^1_0(\Omega) \longrightarrow (H^1(\Omega)/H^1_0(\Omega))^\ast $ defined as
\[ \duality{\Lambda_\sigma f}{g} = \int_\Omega \sigma \nabla u \cdot \nabla v \, \dd x \]
for any $ f, g \in H^1(\Omega)/H^1_0(\Omega) $, where $ u \in H^1(\Omega) $ is the weak solution of the conductivity equation $ \nabla \cdot (\sigma \nabla u) = 0 $ in $ \Omega $ with $ u|_{\partial \Omega} = f $ and $ v \in H^1(\Omega) $ with $ v|_{\partial \Omega} = g $. A significant number of works have been devoted to answering not only the question of uniqueness but also the questions of reconstruction and stability. The most successful approach to treat this problem was introduced by Sylvester and Uhlmann in \cite{SU} and it is based on the construction of complex geometrical optics (CGO) solutions. In dimension $ 2 $, the problem is rather well understood and some important results are \cite{AP}, \cite{ClFR} and \cite{N95}. In dimension greater than $ 2 $, there are still many open questions about the sharp smoothness to ensure uniqueness, stability and reconstruction. Some important results are \cite{HT}, \cite{SU}, \cite{N} and \cite{Al}. Some recent results are \cite{CaGR} and \cite{GZh}. For a more complete list of papers on this problem, we refer the readers to the survey papers \cite{U1} and \cite{U2}.

The literature for the IBVP in electrodynamics under consideration is not as extensive as for Calder\'on's problem. The first partial results are due to Somersalo \textit{et al} in \cite{SIsCh}, for the linearization of the problem at constant electromagnetic parameters, and Sun and Uhlmann in \cite{SuU}, providing a local uniqueness theorem. The first global uniqueness result is due to Ola, P\"aiv\"arinta and Somersalo in \cite{OPS}, where they assume that the electromagnetic coefficients are $ C^3 $-functions and $ \partial \Omega $ is of class $ C^{1, 1} $. They also provided a reconstruction algorithm to recover the coefficients. The arguments in \cite{OPS} are rather complicate since the method developed by Sylvester and Uhlmann in \cite{SU} does not immediately apply. The lack of ellipticity of Maxwell's equations makes the problem more complicate than Calder\'on's. In \cite{OS}, Ola and Somersalo simplified the proof in \cite{OPS} by establishing a relation between Maxwell's equations and a matrix Helmholtz equation with a potential. This relation helps to deal with the lack of ellipticity allowing them to produce exponential growing solutions for Maxwell's equations from the CGOs for the matrix Helmholtz equation. This idea has been extensively used in proving many other results and it will be used in this paper as well. There are other results related to the IBVP under consideration in the literature. Kenig, Salo and Uhlmann proved uniqueness for the corresponding IBVP in some non-euclidean geometries (see \cite{KSU}). With certain type of partial boundary data, the uniqueness was addressed in \cite{CaOSa} by Caro, Ola and Salo (see also \cite{Ca1}). The question of stability has been studied by Caro in \cite{Ca2} assuming full data and in \cite{Ca1} assuming partial data. In \cite{Z}, Zhou used the enclosure method to reconstruct electromagnetic obstacles.

In our paper, Theorem \ref{th:main} lowers significantly the regularity of the coefficients and the  smoothness of the boundary of $\Omega $ assumeded in previous results (despite that domains with Lipschitz boundaries were already considered in \cite{Ca2}) and it matches the one assumed in \cite{HT} for Calder\'on's problem.

The general line of our paper follows the argument in \cite{OS}, relating equation \eqref{eq:t-a_maxwell} with an equation given by a compactly supported zeroth order perturbation of the graded Hodge-Helmholtz operator, namely
\begin{equation}
(\delta d + d \delta - \omega^2 \mu_0 \varepsilon_0) w_j + Q_j w_j = 0, \label{eq:introHEL}
\end{equation}
where $ Q_j = Q(\varepsilon_j + i\sigma_j / \omega, \mu_j, \omega) $ with $ j \in \{ 1, 2 \} $ has to be thought of as a weak potential containing second partial derivatives of $ \mu_j $, $ \varepsilon_j $ and $ \sigma_j $. Using this relation, we are able to prove an integral formula as
\begin{equation}
\Duality{(Q_2 - Q_1) w_1}{v_2} = 0 \label{for:introINTFOR}
\end{equation}
where $ w_1 $ is a solution to \eqref{eq:introHEL} that produces a solution to \eqref{eq:t-a_maxwell} and $ v_2 $ is a solution to a first order elliptic equation (see Section \ref{sec:integral_formula} for more details). This integral formula, with CGOs $ w_1 $ and $ v_2 $ as inputs, will be the starting point of our proof. 

To lower the regularity of the electromagnetic parameters, we adopt a recent improvement of Sylvester and Uhlmann's method that Haberman and Tataru developed in \cite{HT} to prove uniqueness of the Calder\'on problem with continuously differentiable conductivities. For such regularity, solving conductivity equation can be reduced to solving a Schr\"odinger equation $ - \Delta v + m_q v = 0$, where $m_q$ denotes the multiplication operator by the compactly supported weak potential $q=\Delta\sqrt{\sigma}/\sqrt{\sigma}$. Note that this reduction was first used by Sylvester and Uhlmann for smooth conductivities (see \cite{SU}) and later by Brown in \cite{B} for less regular conductivities, all followed by the construction of CGOs in proper function spaces. In \cite{HT}, Haberman and Tataru proved the existence of CGO solutions $v(x) = e^{x\cdot\zeta} (1+\psi_\zeta(x))$ with $\zeta\in\C^n$ and $\zeta\cdot\zeta=0$ to the Schr\"odinger equation. Roughly speaking, the construction is based on solving the equation $ -(\Delta + 2\zeta\cdot\nabla)\psi_\zeta + m_q \psi_\zeta = 0 $ in a Bourgain-type space $\dot{X}^b_\zeta$ whose norm includes the potential $|p_\zeta(\xi)|^{2b}=||\xi|^2-2i\zeta\cdot\xi|^{2b}$ as a weight. In this way, the $\zeta$-dependence is transferred into the space norms and it is shown in \cite{HT} 
\[ \|(\Delta + 2\zeta\cdot\nabla)^{-1}\|_{\dot X^{-1/2}_\zeta\rightarrow\dot X^{1/2}_\zeta}=1, \qquad \|m_q\|_{\dot X^{1/2}_\zeta\rightarrow\dot X^{-1/2}_\zeta}<1, \]
which guarantee the convergence of the Neumann series for $\psi_\zeta$. Furthermore, they obtained an average decaying property for $\|\psi_\zeta\|_{\dot X^{1/2}_\zeta}$, from which they deduced the existence of a sequence $ \{ \zeta^m \} $ such that $ \{ \psi_{\zeta^m} \} $ vanishes as $ m $ grows. 

In this paper, we adopt the idea and several of the estimates in \cite{HT} to construct the CGOs $ w_1 $ and $ v_2 $ with desired properties. Nevertheless, we avoid the argument of extracting the sequence of $ \{ \zeta^m \} $, and use directly the decay in average. This has been previously done in \cite{CaGR} by Caro, Garc\'ia and Reyes to prove stability of the Calder\'on problem for $ C^{1,\epsilon} $-conductivities. When plugging the CGOs $ w_1 $ and $ v_2 $, the output of \eqref{for:introINTFOR} will be certain non-linear relations of $ \varepsilon_1 + i \sigma_1 / \omega $, $ \mu_1 $, $ \varepsilon_2 + i \sigma_2 / \omega $ and $ \mu_2 $ involving second weak partial derivatives of the coefficients. Thus, to conclude the proof of our theorem we will need a unique continuation property for a system of the form
\begin{equation*}
\begin{aligned}
- \Delta f + V f + a f + b g = 0\\
- \Delta g + W g + c g + d f = 0,
\end{aligned}
\end{equation*}
where $ a, b, c $ and $ d $ are compactly supported and belong to $ L^\infty (\R^3) $ while $ V $ and $ W $ are again weak potentials. We will again apply the argument with Bourgain-type spaces to prove the required unique continuation property, which seems not to be available in the literature.

The paper is organized as follows. In Section \ref{sec:auxiliary} we show the relation between \eqref{eq:t-a_maxwell} and \eqref{eq:introHEL}. The proof of the integral formula \eqref{for:introINTFOR} is given in Section \ref{sec:integral_formula}. The CGO solutions are constructed in Section \ref{sec:CGO}, where we will directly refer several times the estimates proven in \cite{HT} rather than listing them in the paper. In Section \ref{sec:uniqueness}, we complete our proof by plugging the CGOs into \eqref{for:introINTFOR} and using the unique continuation principle that we will derive. An appendix is provided at the end of the paper, gathering basic facts and notations in the framework of differential forms, as well including some technical computations for the electromagnetic IBVP.

\begin{acknowledgments} This work was initiated while the authors were visiting Gunther Uhlmann at UCI. The authors would like to thank him for his generosity, hospitality and many useful discussions. The visit was partially supported by the Department of Mathematics of UCI and by the organizing committee of ``A conference on inverse problems in honor of Gunther Uhlmann'' (Irvine, June 2012). PC wants to thank Petri Ola and Mikko Salo for useful discussions. He is supported by ERC-2010 Advanced Grant, 267700 - InvProb and belongs to MTM 2011-02568. TZ is partly supported by NSF grant DMS1161129.
\end{acknowledgments}

\section{An auxiliary graded equation} \label{sec:auxiliary}
In this section we establish a relation between
\begin{equation*}
\delta u^2 + i \omega \varepsilon u^1 - d u^1 + i \omega \mu u^2 = \sigma u^1
\end{equation*}
and an auxiliary graded Hodge-Helmholtz equation with zeroth order perturbation (following the idea in \cite{OS}), which allows the construction of CGOs. For our purposes, it would be enough to have solutions in $ \Omega $, but for convenience we will conduct our analysis in the whole $ \R^3 $. This gives us certain freedom in extending the coefficients outside $ \Omega $. Thus, set $ B = \{ x \in \R^3 : |x| < R \} $ with $ R>0 $ such that $ \overline{\Omega} \subset B $. Let $ \omega $, $ \mu_0 $ and $ \varepsilon_0 $ be three positive constants. At this point, we consider $ \mu, \varepsilon $ and $ \sigma $ in $ W^{1, \infty} (\R^3) $, the space of measurable functions modulo those vanishing almost everywhere such that themselves and their first weak partial derivatives are essentially bounded in $ \R^3 $. Furthermore, we assume that $ \mu, \varepsilon $ and $ \sigma $ are real-valued,
\[ \supp (\mu - \mu_0) \subset B,\quad \supp (\varepsilon - \varepsilon_0) \subset B,\quad \supp (\sigma) \subset B \]
and $ \mu(x) \geq \mu_0 $, $ \varepsilon(x) \geq \varepsilon_0 $, $ \sigma(x) \geq 0 $ for almost every $ x $ in $ \R^3 $. For simplicity, write $ \gamma = \varepsilon + i \sigma / \omega $. It is sufficient for us, to produce weak solutions to
\begin{equation}
\delta u^2 + i \omega \gamma u^1 - d u^1 + i \omega \mu u^2 = 0 \label{eq:t-h_maxwell_R3}
\end{equation}
in $ \R^3 $, namely, forms $ u^1 + u^2 $ with $ u^l \in L^2_\mathrm{loc} (\R^3; \Lambda^l \R^3) $ satisfying
\[ \Duality{\delta u^2 + i \omega \gamma u^1 - d u^1 + i \omega \mu u^2}{\varphi^1 + \varphi^2} = 0 \]
for all $ \varphi^1 + \varphi^2 $ with $ \varphi^l \in C^\infty_0 (\R^3;\Lambda^l \R^3) $. Here $ \duality{\centerdot}{\centerdot} $ denotes the duality bracket for distributions.

In order to derive the auxiliary equation, we augment \eqref{eq:t-h_maxwell_R3} by adding
\[ - \gamma^{-1} \delta (\gamma u^1) + \mu^{-1} d (\mu u^2) = 0, \]
which is derived directly from \eqref{eq:t-h_maxwell_R3}. 

Next, we consider an equation of the graded form $\sum_{l=0}^3u^l$ where $u^l\in L^2_{\mathrm{loc}}(\R^3;\Lambda^l\R^3)$
\begin{gather*}
- \gamma^{-1} \delta (\gamma u^1) + i \omega \mu u^0 + \mu^{-1} d (\mu u^0) + \delta u^2 + i \omega \gamma u^1 \\
- \gamma^{-1} \delta (\gamma u^3) - d u^1 + i \omega \mu u^2 + \mu^{-1} d (\mu u^2) + i \omega \gamma u^3 = 0.
\end{gather*}
Multiplying $0, 2$-forms by $\gamma^{1/2}$ and $1, 3$-forms by $\mu^{1/2}$,
we obtain
\begin{gather*}
- \gamma^{-1/2} \delta (\gamma u^1) + i \omega \gamma^{1/2} \mu u^0 + \mu^{-1/2} d (\mu u^0) + \mu^{1/2} \delta u^2 + i \omega \gamma \mu^{1/2} u^1 \\
- \gamma^{-1/2} \delta (\gamma u^3) - \gamma^{1/2} d u^1 + i \omega \gamma^{1/2} \mu u^2 + \mu^{-1/2} d (\mu u^2) + i \omega \gamma \mu^{1/2} u^3 = 0.
\end{gather*}
Throughout this paper $ (\centerdot)^{1/2} $ will denote the principal branch of the square root. 
(Same convention will apply to $\log(\centerdot)$)
If we now set
\[v=\sum_{l = 0}^3 v^l = \mu^{1/2} u^0 + \gamma^{1/2} u^1 + \mu^{1/2} u^2 + \gamma^{1/2} u^3,\]

we end up with the equation
\begin{equation}
P(d + \delta; \gamma, \mu, \omega)v = 0,
\label{eq:rescaled-equation}
\end{equation}
where
\begin{equation*}
\begin{split}
P(d + \delta; \gamma, \mu, \omega)v = &(d + \delta) \sum_{l = 0}^3 (-1)^l v^l + d a \wedge v^1 + d a \vee (v^1 + v^3) \\
&+ d b \wedge (v^0 + v^2) - d b \vee v^2 + i \omega \gamma^{1/2} \mu^{1/2} v,
\end{split}
\end{equation*}
$ a = \frac{1}{2} \log \gamma $ and $ b = \frac{1}{2} \log \mu $. 
The key point of this derivation is to note that $ v = \sum_0^3 v^l $ with $ v^l \in L^2_\mathrm{loc} (\R^3; \Lambda^l \R^3) $ is a weak solution of \eqref{eq:rescaled-equation} in $ \R^3 $ (that is, for every $\varphi = \sum_0^3\varphi^l$ with $\varphi^l\in C^\infty_0(\R^3;\Lambda^l\R^3)$,
\[ \Duality{P(d + \delta; \gamma, \mu, \omega)v}{\varphi} = 0 \]
with $ \duality{\centerdot}{\centerdot} $ denoting the duality bracket for distributions) and $ v^0 + v^3 = 0 $ if, and only if, $ u^1 + u^2 = \gamma^{-1/2} v^1 + \mu^{-1/2} v^2 $ with $ u^l \in L^2_\mathrm{loc} (\R^3; \Lambda^l \R^3) $ is a weak solution of \eqref{eq:t-h_maxwell_R3} in $ \R^3 $. For convenience, let us define operator
\begin{equation}\label{eq:Pt}\begin{split}
P(d + \delta; \gamma, \mu, \omega)^t w: = &(d + \delta) \sum_{l = 0}^3 (-1)^{l + 1} w^l + d b \wedge w^1 + d b \vee (w^1 + w^3)\\
&+d a \wedge (w^0 + w^2) - d a \vee w^2 + i \omega \gamma^{1/2} \mu^{1/2} w
\end{split}\end{equation}
for $ w = \sum_0^3 w^l $ with $ w^l \in H^\delta_\mathrm{loc} (\R^3; \Lambda^l \R^3) \cap H^d_\mathrm{loc} (\R^3; \Lambda^l \R^3) $. Note that $ P(d + \delta; \gamma, \mu, \omega)^t $ is the formally transpose of $ P(d + \delta; \gamma, \mu, \omega) $. 

Due to the rescaling by $\gamma^{1/2}$ and $\mu^{1/2}$ that we chose,
it can be verified that $ P(d + \delta; \gamma, \mu, \omega) \circ P(d + \delta; \gamma, \mu, \omega)^t $ is a zeroth order perturbation of the graded Hodge-Helmholtz operator. Set, for any graded forms $ w = \sum_0^3 w^l $ and $ \varphi = \sum_0^3 \varphi^l $ with $ w^l, \varphi^l \in H^1_\mathrm{loc} (\R^3; \Lambda^l \R^3) $,
\begin{equation}\label{eq:multi-Q}
\begin{gathered}
\Duality{Q (\gamma, \mu, \omega) w}{\varphi} = - \int_{\R^3} \omega^2 (\gamma \mu - \varepsilon_0 \mu_0) \inner{w}{\varphi} \, \dd x  \\
+ \int_{\R^3} \inner{i2\omega d( \gamma^{1/2} \mu^{1/2}) \vee (w^1 + w^3) + i2\omega d( \gamma^{1/2} \mu^{1/2}) \wedge (w^0 + w^2)}{\varphi} \, \dd x \\
+ \int_{\R^3} \inner{d a}{d a} \inner{w^0 + w^2}{\varphi^0 + \varphi^2} + \inner{d b}{d b} \inner{w^1 + w^3}{\varphi^1 + \varphi^3} \, \dd x \\
+ \int_{\R^3} \inner{d a}{d \inner{- w^0 + w^2}{\varphi^0 + \varphi^2}} + \inner{d b}{d \inner{w^1 - w^3}{\varphi^1 + \varphi^3}} \, \dd x \\
+ \int_{\R^3} \inner{d b}{D^\ast (w^1 \odot \varphi^1)} \, \dd x + \int_{\R^3} \inner{d a}{D^\ast (\ast w^2 \odot \ast \varphi^2) } \, \dd x .
\end{gathered}
\end{equation}
\begin{proposition} \label{prop:HL+Q} \sl
Let $ w = \sum_0^3 w^l $ be a graded form with $ w^l \in H^1_\mathrm{loc} (\R^3; \Lambda^l \R^3) $ and assume that
\begin{equation}\label{eq:Helm}
\int_{\R^3} \inner{\delta w}{\delta \varphi} + \inner{d w}{d \varphi} - \omega^2 \varepsilon_0 \mu_0 \inner{w}{\varphi} \, \dd x + \Duality{Q (\gamma, \mu, \omega) w}{\varphi} = 0
\end{equation}
for all $ \varphi = \sum_0^3 \varphi^l $ with $ \varphi^l \in C^\infty_0 (\R^3;\Lambda^l \R^3) $. Then, $ v = \sum_0^3 v^l $ defined by 
\begin{equation}
v = P(d + \delta; \gamma, \mu, \omega)^t w \label{def:v=Pw}
\end{equation}
is a weak solution to \eqref{eq:rescaled-equation} in $ \R^3 $ and $ v^l \in H^1_\mathrm{loc} (\R^3; \Lambda^l \R^3) $. 
\end{proposition}
\begin{proof}
We first prove that $ v $ is a weak solution to \eqref{eq:rescaled-equation}. Since $ v^l \in L^2_\mathrm{loc} (\R^3; \Lambda^l \R^3) $, it is enough to show that 
\begin{equation}
\begin{gathered}
\int_{\R^3}\inner{P(d + \delta; \gamma, \mu, \omega)^t w}{P(d + \delta; \gamma, \mu, \omega)^t \varphi} \, \dd x\\
= \int_{\R^3} \inner{\delta w}{\delta \varphi} + \inner{d w}{d \varphi} - \omega^2 \varepsilon_0 \mu_0 \inner{w}{\varphi} \, \dd x + \Duality{Q (\gamma, \mu, \omega) w}{\varphi}.
\end{gathered}
\label{id:giving_HL+Q}
\end{equation}
To check this, by direct computations the first four terms in the left hand side are 
\begin{equation}\label{eq:firstterm}
\begin{gathered}
\int_{\R^3} \Binner{(d + \delta) \sum_{l = 0}^3 (-1)^{l + 1} w^l}{(d + \delta) \sum_{l = 0}^3 (-1)^{l + 1} \varphi^l} \, \dd x \\
= \int_{\R^3} \inner{\delta w}{\delta \varphi} + \inner{d w}{d \varphi} \, \dd x,
\end{gathered}
\end{equation}
\begin{equation}
\int_{\R^3} \inner{i \omega \gamma^{1/2} \mu^{1/2} w}{i \omega \gamma^{1/2} \mu^{1/2} \varphi} \, \dd x = - \int_{\R^3} \omega^2 \gamma \mu \inner{w}{\varphi} \, \dd x,
\end{equation}
\begin{equation}
\begin{gathered}
\int_{\R^3} \Binner{(d + \delta) \sum_{l = 0}^3 (-1)^{l + 1} w^l}{i \omega \gamma^{1/2} \mu^{1/2} \varphi} \, \dd x \\
+ \int_{\R^3} \Binner{i \omega \gamma^{1/2} \mu^{1/2} w}{(d + \delta) \sum_{l = 0}^3 (-1)^{l + 1} \varphi^l} \, \dd x \\
= \int_{\R^3} \inner{i\omega d( \gamma^{1/2} \mu^{1/2}) \vee (w^1 + w^3) + i\omega d( \gamma^{1/2} \mu^{1/2}) \wedge (w^0 + w^2)}{\varphi} \, \dd x \\
+ \int_{\R^3} \inner{i\omega d( \gamma^{1/2} \mu^{1/2}) \vee w^2 - i\omega d( \gamma^{1/2} \mu^{1/2}) \wedge w^1}{\varphi} \, \dd x
\end{gathered}
\end{equation}
and
\begin{equation}\begin{gathered}
\int_{\R^3} \langle d b \wedge w^1 + d b \vee (w^1 + w^3) + d a \wedge (w^0 + w^2) - d a \vee w^2, i \omega \gamma^{1/2} \mu^{1/2} \varphi \rangle \\
+ \langle i \omega \gamma^{1/2} \mu^{1/2} w, d b \wedge \varphi^1 + d b \vee (\varphi^1 + \varphi^3) + d a \wedge (\varphi^0 + \varphi^2) - d a \vee \varphi^2 \rangle \, \dd x \\
= \int_{\R^3} \inner{i\omega d( \gamma^{1/2} \mu^{1/2}) \vee (w^1 + w^3) + i\omega d( \gamma^{1/2} \mu^{1/2}) \wedge (w^0 + w^2)}{\varphi} \, \dd x \\
- \int_{\R^3} \inner{i\omega d( \gamma^{1/2} \mu^{1/2}) \vee w^2 - i\omega d( \gamma^{1/2} \mu^{1/2}) \wedge w^1}{\varphi} \, \dd x.
\end{gathered}\end{equation}
By Corollary \ref{cor:innerinner}, the fifth term gives
\begin{equation}
\begin{gathered}
\int_{\R^3} \langle d b \wedge w^1 + d b \vee (w^1 + w^3) + d a \wedge (w^0 + w^2) - d a \vee w^2, \\
 d b \wedge \varphi^1 + d b \vee (\varphi^1 + \varphi^3) + d a \wedge (\varphi^0 + \varphi^2) - d a \vee \varphi^2 \rangle \, \dd x \\
= \int_{\R^3} \inner{d a}{d a} \inner{w^0 + w^2}{\varphi^0 + \varphi^2} + \inner{d b}{d b} \inner{w^1 + w^3}{\varphi^1 + \varphi^3} \, \dd x.
\end{gathered}
\end{equation}
By Proposition \ref{prop:vdwdelta} the last term yields
\begin{equation}\label{eq:lastterm}
\begin{gathered}
\int_{\R^3} \Big \langle d b \wedge w^1 + d b \vee (w^1 + w^3) + d a \wedge (w^0 + w^2) - d a \vee w^2, \\
(d + \delta) \sum_{l = 0}^3 (-1)^{l + 1} \varphi^l \Big \rangle + \Big \langle (d + \delta) \sum_{l = 0}^3 (-1)^{l + 1} w^l, \\
d b \wedge \varphi^1 + d b \vee (\varphi^1 + \varphi^3) + d a \wedge (\varphi^0 + \varphi^2) - d a \vee \varphi^2 \rangle \, \dd x = \\
= \int_{\R^3} \inner{d a}{d \inner{- w^0 + w^2}{\varphi^0 + \varphi^2}} + \inner{d b}{d \inner{w^1 - w^3}{\varphi^1 + \varphi^3}} \, \dd x \\
+ \int_{\R^3} \inner{d b}{D^\ast (w^1 \odot \varphi^1)} \, \dd x + \int_{\R^3} \inner{d a}{D^\ast (\ast w^2 \odot \ast \varphi^2) } \, \dd x = 0.
\end{gathered}
\end{equation}
Summing up identities \eqref{eq:firstterm} through \eqref{eq:lastterm} gives identity \eqref{id:giving_HL+Q}.

It remains to prove that $ v^l \in H^1_\mathrm{loc} (\R^3; \Lambda^l \R^3) $. Since $ v^l \in L^2_\mathrm{loc} (\R^3; \Lambda^l \R^3) $, we have $ (d + \delta) \sum_0^3 (-1)^l v^l \in \bigoplus_0^3 L^2_\mathrm{loc} (\R^3; \Lambda^l \R^3) $ by \eqref{eq:rescaled-equation}. Therefore, Lemma \ref{lem:H^delta-H^d} allows to conclude the proof.
\end{proof}
\begin{remark} \label{rem:giving_HL+Q} \sl Identity \eqref{id:giving_HL+Q} holds even for $ \varphi = \sum_0^3 \varphi^l $ with $ \varphi^l \in H^1_\mathrm{loc} (\R^3; \Lambda^l \R^3) $.
\end{remark}

Similar calculation verifies that the same property holds for $ P(d + \delta; \gamma, \mu, \omega)^t \circ P(d + \delta; \gamma, \mu, \omega) $ as stated in Proposition \ref{prop:HL+tQ}. 
Define
\begin{equation}
\label{eq:multi-tQ}
\begin{gathered}
\Duality{\tilde{Q} (\gamma, \mu, \omega) w}{\varphi} = - \int_{\R^3} \omega^2 (\gamma \mu - \varepsilon_0 \mu_0) \inner{w}{\varphi} \, \dd x  \\
+ \int_{\R^3} \inner{-i2\omega d( \gamma^{1/2} \mu^{1/2}) \vee w^2 + i2\omega d( \gamma^{1/2} \mu^{1/2}) \wedge w^1}{\varphi} \, \dd x \\
+ \int_{\R^3} \inner{d b}{d b} \inner{w^0 + w^2}{\varphi^0 + \varphi^2} + \inner{d a}{d a} \inner{w^1 + w^3}{\varphi^1 + \varphi^3} \, \dd x \\
+ \int_{\R^3} \inner{d b}{d \inner{w^0 - w^2}{\varphi^0 + \varphi^2}} + \inner{d a}{d \inner{-w^1 + w^3}{\varphi^1 + \varphi^3}} \, \dd x \\
- \int_{\R^3} \inner{d a}{D^\ast (w^1 \odot \varphi^1)} \, \dd x - \int_{\R^3} \inner{d b}{D^\ast (\ast w^2 \odot \ast \varphi^2) } \, \dd x = 0,
\end{gathered}
\end{equation}
for $ w = \sum_0^3 w^l $ and $ \varphi = \sum_0^3 \varphi^l $ with $ w^l, \varphi^l \in H^1_\mathrm{loc} (\R^3; \Lambda^l \R^3) $.
\begin{proposition} \label{prop:HL+tQ} \sl
Let $ w = \sum_0^3 w^l $ be a graded form with $ w^l \in H^1_\mathrm{loc} (\R^3; \Lambda^l \R^3) $ and assume that
\begin{equation}
\int_{\R^3} \inner{\delta w}{\delta \varphi} + \inner{d w}{d \varphi} - \omega^2 \varepsilon_0 \mu_0 \inner{w}{\varphi} \, \dd x + \Duality{\tilde{Q} (\gamma, \mu, \omega) w}{\varphi} = 0 \label{eq:HlmtQ}
\end{equation}
for all $ \varphi = \sum_0^3 \varphi^l $ with $ \varphi^l \in C^\infty_0 (\R^3;\Lambda^l \R^3) $. Then, $ v = \sum_0^3 v^l $ defined by 
\begin{equation*}
v = P(d + \delta; \gamma, \mu, \omega) w
\end{equation*}
is a weak solution of
\[ P(d + \delta; \gamma, \mu, \omega)^t v = 0 \]
in $ \R^3 $ and $ v^l \in H^1_\mathrm{loc} (\R^3; \Lambda^l \R^3) $. 
\end{proposition}

Recall that $ v = \sum_0^3 v^l $ with $ v^l \in L^2_\mathrm{loc} (\R^3; \Lambda^l \R^3) $ is a weak solution to \eqref{eq:rescaled-equation} and satisfies $ v^0 + v^3 = 0 $ in $ \R^3 $ if and only if that $ u^1 + u^2 = \gamma^{-1/2} v^1 + \mu^{-1/2} v^2 $ with $ u^l \in L^2_\mathrm{loc} (\R^3; \Lambda^l \R^3) $ is a weak solution of \eqref{eq:t-h_maxwell_R3} in $ \R^3 $. We finish this section by singling out the equation of $ v^0 + v^3 $ from \eqref{eq:HlmtQ}, which is used later to show the CGOs we will construct in Section \ref{sec:CGO} satisfy $ v^0 + v^3 = 0 $.
\begin{proposition}\label{prop:v0v3eq0} \sl Let $ v = \sum_0^3 v^l $ with $ v^l \in H^1_\mathrm{loc} (\R^3; \Lambda^l \R^3) $ satisfy
\[ P(d + \delta; \gamma, \mu, \omega) v = 0 \]
in any bounded open subset of $ \R^3 $, then for any $ \varphi = \varphi^0 + \varphi^3 $ with $ \varphi^l$ belonging $ C^\infty_0 (\R^3; \Lambda^l \R^3) $ we have that
\begin{equation}\label{eq:v0v3-Sch}\begin{gathered}
\int_{\R^3} \inner{\delta (v^0 + v^3)}{\delta \varphi} + \inner{d (v^0 + v^3)}{d \varphi} - \omega^2 \varepsilon_0 \mu_0 \inner{v^0 + v^3}{\varphi} \, \dd x \\
+\Duality{\tilde{q}(\gamma, \mu, \omega) (v^0 + v^3)}{\varphi} = 0,
\end{gathered}\end{equation}
where
\begin{gather*}
\Duality{\tilde{q}(\gamma, \mu, \omega) (v^0 + v^3)}{\varphi} = - \int_{\R^3} \omega^2 (\gamma \mu - \varepsilon_0 \mu_0) \inner{v^0 + v^3}{\varphi} \, \dd x \\
+ \int_{\R^3} \inner{d b}{d b} \inner{v^0}{\varphi^0} + \inner{d a}{d a} \inner{v^3}{\varphi^3} + \inner{d b}{d\inner{v^0}{\varphi^0}} + \inner{d a}{d\inner{v^3}{\varphi^3}} \, \dd x.
\end{gather*}
\end{proposition}
\begin{proof} This is immediate from the proof of Proposition \ref{prop:HL+tQ} and the fact that $ \tilde{Q} (\gamma, \mu, \omega) $ decouples for $ v^0 + v^3 $.
\end{proof}

\section{An integral formula} \label{sec:integral_formula}
In this section we provide an integral formula that serves as the starting point to prove uniqueness of the IBVP. To do this, we exploit the computations which allow to produce solutions for \eqref{eq:t-h_maxwell_R3} from solutions of \eqref{eq:Helm} (see Proposition \ref{prop:HL+Q} in Section \ref{sec:auxiliary}). 

Let $ \Omega $ be a bounded non-empty open subset in $ \R^3 $ whose boundary $ \partial \Omega $ can be locally described by the graph of a Lipschitz function. Throughout the rest of the paper, we assume that $ \mu_j, \varepsilon_j, \sigma_j $ belong to $ C^1 (\overline{\Omega}) $ with $ j \in \{ 1, 2 \} $ such that $ \mu_j(x) \geq \mu_0 $, $ \varepsilon_j(x) \geq \varepsilon_0 $, $ \sigma_j(x) \geq 0 $ everywhere in $ \Omega $. Here we say that $ f $ is in $ C^1 (\overline{\Omega}) $ if $ f : \Omega \longrightarrow \C $ is continuously differentiable in $ \Omega $ and its partial derivatives $ \partial^\alpha f $ are uniformly continuous in $ \Omega $ for $ \alpha \in \N^3 $ and $ |\alpha| = 1 $ and
\begin{equation}
|\partial^\alpha f (x)| \leq C, \qquad \forall x\in \Omega, \quad |\alpha|\leq 1, \label{es:C1BOUNDNESS}
\end{equation}
for certain positive constant $ C $. The norm on $ C^1 (\overline{\Omega}) $, defined as the smallest constant $ C $ for which \eqref{es:C1BOUNDNESS} hold, makes $ C^1 (\overline{\Omega}) $ a Banach space. Since $ \partial \Omega $ is of Lipschitz class, $ f $ defined as above is uniformly continuous and, consequently, $ \partial^\alpha f $ possesses a unique bounded continuous extension to $ \overline{\Omega} $ for any $ |\alpha| \leq 1 $. This extension will be still denoted by $ f $.

Consider $ C_j=C(\mu_j,\varepsilon_j,\sigma_j; \omega) $ the Cauchy data set associated to $ \mu_j, \varepsilon_j, \sigma_j $ at frequency $ \omega >0$. Write $ \gamma_j = \varepsilon_j + i \sigma_j / \omega $ and assume $ \partial^\alpha \gamma_1 (x) = \partial^\alpha \gamma_2 (x) $ and $ \partial^\alpha \mu_1 (x) = \partial^\alpha \mu_2 (x) $ for all $ x \in \partial \Omega $ and $ |\alpha| \leq 1 $. We can extend\footnote{The extensions we want to perform here are of Whitney type. These kind of extensions hold for functions defined on any closed subset of $ \R^n $ whenever the functions can be approximeted by certain polynomials. In order to ensure the existence of such polynomials, we use that $  \partial \Omega $ is of Lipschitz class. The argument to prove the existence of such polynomials is similar to the one carried out in Section 2 of \cite{CaGR} for $ C^{1,\varepsilon}(\overline{\Omega}) $ functions with the only difference that, where the authors referred to Chapter VI \S2 of \cite{St}, we refer to Chapter VI \S4.7 of \cite{St}.} $ \gamma_j $ and $ \mu_j $ to continuously differentiable functions in $\R^3$, still denoted by $ \gamma_j $ and $ \mu_j $, such that $ |\partial^\alpha \gamma_j (x)| + |\partial^\alpha \mu_j (x)| \leq C $, $ \mu_j(x) \geq \mu_0 $, $ \varepsilon_j(x) \geq \varepsilon_0 $, $ \sigma_j(x) \geq 0 $ for all $ x \in \R^3 $, $ |\alpha| \leq 1 $ and certain constant $ C > 0 $, 
\[ \supp (\mu_j - \mu_0) \subset B,\qquad \supp (\gamma_j - \varepsilon_0) \subset B \]
where $ B = \{ x \in \R^3 : |x| < R \} \supset \overline{\Omega} $, and $ \gamma_1 (x) = \gamma_2 (x) $ and $ \mu_1 (x) = \mu_2 (x) $ for all $ x \in \R^3 \setminus \Omega $. For convenience, we write $ a_j = \frac{1}{2} \log \gamma_j $ and $ b_j = \frac{1}{2} \log \mu_j $.

\begin{proposition}\label{prop:int-formula} \sl
Let $ w_1 = \sum_0^3 w^l_1 $ be a graded form with $ w^l_1 \in H^1_\mathrm{loc} (\R^3; \Lambda^l \R^3) $
satisfying
\begin{equation}
\int_{\R^3} \inner{\delta w_1}{\delta \varphi} + \inner{d w_1}{d \varphi} - \omega^2 \varepsilon_0 \mu_0 \inner{w_1}{\varphi} \, \dd x + \Duality{Q(\gamma_1, \mu_1, \omega) w_1}{\varphi} = 0, \label{eq:HL+Q1}
\end{equation}
for all $ \varphi = \sum_0^3 \varphi^l $ with $ \varphi^l \in C^\infty_0 (\R^3;\Lambda^l \R^3) $. Assume that $ v_1 = \sum_0^3 v_1^l $, defined by
\begin{equation}\label{eq:v1ofw1}
v_1 = P(d + \delta ;\gamma_1, \mu_1, \omega)^t w_1,
\end{equation}
satisfies $ v_1^0 + v_1^3 = 0 $. Let $ v_2 = \sum_0^3 v_2^l $ with $ v_2^l \in H^1_\mathrm{loc} (\R^3; \Lambda^l \R^3) $ satisfy
\begin{equation}\label{eq:Ptv2=0}
P(d + \delta ;\gamma_2, \mu_2, \omega)^t v_2 = 0
\end{equation}
in any bounded open subset of $ \R^3 $. Then $ C_1 = C_2 $ implies
\[ \Duality{(Q(\gamma_2, \mu_2, \omega) - Q(\gamma_1, \mu_1, \omega)) w_1}{v_2} = 0. \]
\end{proposition}
\begin{proof}
By Remark \ref{rem:giving_HL+Q} and because $ \gamma_1 (x) = \gamma_2 (x) $ and $ \mu_1 (x) = \mu_2 (x) $ for all $ x \in \R^3 \setminus \Omega $, we know that
\begin{gather*}
\Duality{(Q(\gamma_2, \mu_2, \omega) - Q(\gamma_1, \mu_1, \omega)) w_1}{v_2} \\
= \int_\Omega\inner{P(d + \delta; \gamma_2, \mu_2, \omega)^t w_1}{P(d + \delta; \gamma_2, \mu_2, \omega)^t v_2} \, \dd x \\
- \int_\Omega\inner{P(d + \delta; \gamma_1, \mu_1, \omega)^t w_1}{P(d + \delta; \gamma_1, \mu_1, \omega)^t v_2} \, \dd x\\
= - \int_\Omega\inner{v_1}{P(d + \delta; \gamma_1, \mu_1, \omega)^t v_2} \, \dd x.
\end{gather*}
The last equality follows from \eqref{eq:Ptv2=0} and \eqref{eq:v1ofw1}.

Since $ v^0_1 + v^3_1 = 0 $, we have that $ u^1_1 + u^2_1 = \gamma_1^{-1/2} v^1_1 + \mu_1^{-1/2} v^2_1 $ satisfies
\begin{equation}
\delta u^2_1 + i \omega \gamma_1 u^1_1 - d u^1_1 + i \omega \mu_1 u^2_1 = 0 \label{eq:maxwell1}
\end{equation}
almost everywhere in $ \Omega $ (see Section \ref{sec:auxiliary}). The definitions of boundary traces $ \delta \tr $ and $ d \tr $ (see Appendix \ref{app:traces}) give
\begin{gather}
- \int_\Omega\inner{v_1}{P(d + \delta; \gamma_1, \mu_1, \omega)^t v_2} \, \dd x = \Duality{\delta \tr (\gamma_1 u^1_1)}{\gamma_1^{-1/2} v^0_2} + \Duality{\delta \tr\, u^2_1}{\mu^{1/2}_1 v^1_2}\label{eq:pf-prop3.1-IBP-1} \\
- \Duality{d \tr\, u^1_1}{\gamma_1^{1/2} v^2_2} + \Duality{d \tr (\mu_1 u^2_1)}{\mu_1^{- 1/2} v^3_2}.\nonumber
\end{gather}

Suppose $ f = f^1 + f^2 $ with $ f^l \in L^2_\mathrm{loc}(\R^3; \Lambda^l \R^3) $ is a weak solution to
\begin{equation}
\delta f^2 + i \omega \gamma_2 f^1 - d f^1 + i \omega \mu_2 f^2 = 0 \label{eq:maxwell2}
\end{equation}
in $ \R^3 $. Note that then $ f^1 \in H^d(\Omega; \Lambda^1 \R^3) $, $ f^2 \in H^\delta (\Omega; \Lambda^2 \R^3) $. Set $ g = g^1 + g^2 = \gamma_2^{1/2} f^1 + \mu_2^{1/2} f^2 $. By \eqref{eq:Ptv2=0}, we obviously have
\[ \int_\Omega\inner{g}{P(d + \delta; \gamma_2, \mu_2, \omega)^t v_2} \, \dd x = 0. \]
Once more by the definitions of $ \delta \tr $ and $ d \tr $, we have
\begin{gather}
0 = \int_\Omega\inner{g}{P(d + \delta; \gamma_2, \mu_2, \omega)^t v_2} \, \dd x = - \Duality{\delta \tr (\gamma_2 f^1)}{\gamma_2^{-1/2} v^0_2} - \Duality{\delta \tr\, f^2}{\mu^{1/2}_2 v^1_2} \label{eq:pf-prop3.1-IBP-2}\\
+ \Duality{d \tr\, f^1}{\gamma_2^{1/2} v^2_2} - \Duality{d \tr (\mu_2 f^2)}{\mu_2^{- 1/2} v^3_2}.\nonumber
\end{gather}

Since $ \delta \tr u^2_1 + d \tr u^1_1 \in C_1 = C_2 $ by assumption, there exists $ u_2 = u^1_2 + u^2_2 $ with $ u^1_2 \in H^d(\Omega; \Lambda^1 \R^3) $ and $ u^2_2 \in H^\delta (\Omega; \Lambda^2 \R^3) $ a solution to \eqref{eq:maxwell2} in $ \Omega $ such that 
\[\delta \tr u^2_1 + d \tr u^1_1 = \delta \tr u^2_2 + d \tr u^1_2.\] 
Define\footnote{This definition satisfies the appropriate conditions since $ \gamma_1 (x) = \gamma_2 (x) $ and $ \mu_1 (x) = \mu_2 (x) $ for all $ x \in \R^3 \setminus \Omega $.} $ f (x) = u_2 (x) $ for almost every $ x \in \Omega $ and $ f (x) = u_1 (x) $ for almost every $ x \in \R^3 \setminus \Omega $. Using \eqref{eq:pf-prop3.1-IBP-1} and \eqref{eq:pf-prop3.1-IBP-2} and noting that $ \gamma_1 (x) = \gamma_2 (x) $ and $ \mu_1 (x) = \mu_2 (x) $ for all $ x \in \partial \Omega $, we can conclude
\begin{gather*}
\Duality{(Q(\gamma_2, \mu_2, \omega) - Q(\gamma_1, \mu_1, \omega)) w_1}{v_2} = - \frac{1}{i\omega} \Duality{\delta \tr (\delta u^2_1)}{\gamma_2^{-1/2} v^0_2} \\
+ \frac{1}{i\omega} \Duality{d \tr (d u^1_1)}{\mu_2^{- 1/2} v^3_2} + \frac{1}{i\omega} \Duality{\delta \tr (\delta u_2^2)}{\gamma_2^{-1/2} v^0_2} - \frac{1}{i\omega} \Duality{d \tr (d u_2^1)}{\mu_2^{- 1/2} v^3_2}.
\end{gather*}
The result follows by Lemma \ref{lem:traces}.
\end{proof}

\section{The construction of CGO solutions}\label{sec:CGO}
In this section we construct the CGO solutions that will be plugged into the integral formula in Proposition \ref{prop:int-formula}. To deal with less regular electromagnetic coefficients than that of \cite{OS}, we adopt Bourgain-type spaces introduced by Haberman and Tataru in \cite{HT}. 

Let $ \zeta = \sum_1^3 \zeta_j dx^j $ be a constant $ 1 $-differential form in $ \R^3 $ and let $ p_\zeta $ denote the polynomial 
\[ p_\zeta (\xi) = |\xi|^2 - 2i \inner{\zeta}{\xi} .\] 
For any $ b \in \R $, let $\dot{X}^b_\zeta$ denote the space of graded forms $w=\sum_{0}^3w^l$ such that $w^l\in\sh'(\R^3;\Lambda^l\R^3)$ and its Fourier transform
\[\widehat{w^l}\in L^2(\R^3, |p_\zeta|^{2b} \dd \xi;\Lambda^l\R^3).\]
The functional
\[ w \in \dot{X}^b_\zeta \longmapsto \norm{w}{}{\dot{X}^b_\zeta} = \left( \sum_{l = 0}^3 \norm{|p_\zeta|^b \widehat{w^l}}{2}{L^2(\R^3; \Lambda^l \R^3)} \right)^{1/2} \]
makes $ \dot{X}^b_\zeta $ a normed space. Moreover, if $ b < 1 $, then $ \dot{X}^b_\zeta $ is a Hilbert space. As in \cite{HT}, we will only use the cases where $ b \in \{ 1/2, -1/2 \} $. Note that $ \dot{X}_\zeta^{-1/2} $ can be identified as the dual space of $ \dot{X}_\zeta^{1/2} $. The simplest feature of these spaces is that the operator $ (\Delta_\zeta+ \inner{\zeta}{\zeta})^{-1} $  (defined by the symbol $(p_\zeta)^{-1}$) is a bounded linear operator from $ \dot{X}_\zeta^{-1/2} $ to $ \dot{X}_\zeta^{1/2} $ with norm
\begin{equation}
\norm{(\Delta_\zeta+\inner{\zeta}{\zeta})^{-1}}{}{\dot{X}_\zeta^{-1/2}\rightarrow \dot{X}_\zeta^{1/2}}=1. \label{id:Dz-k2=1}
\end{equation}
Let $\Delta_\zeta$ denote the conjugate operator $ \Delta_\zeta = e_{-\zeta} (d\delta+\delta d) \circ e_\zeta $ where $ e_\zeta (x) = e^{\zeta \cdot x} $ and $ \zeta \cdot x = \sum_{1}^3 \zeta_j x^j $.
\begin{remark} \label{rem:uniqueness} \sl Given $ f \in \dot{X}^{-1/2}_{\zeta} $, it is an obvious consequence of the definition of $ \dot{X}^{1/2}_{\zeta} $ that there exists a unique $ u \in \dot{X}^{1/2}_{\zeta} $ satisfying
\begin{equation*}
\Delta_{\zeta} u + \inner{\zeta}{\zeta}u = f.
\end{equation*}
\end{remark}
\begin{remark} \label{rem:smoothness} \sl If $ u \in \dot{X}^{1/2}_\zeta $ with $ u = \sum_0^3 u^l $ then $ u^l \in H^1_\mathrm{loc} (\R^3; \Lambda^l \R^3) $. This is a simple consequence of (5) and (6) in Lemma 2.2 of \cite{HT} and the finite band property (sometimes called Bernstein's inequality).
\end{remark}

\subsection{The construction of $ w_1 $.} Let $ \zeta_1 $ be a complex-valued constant $ 1 $-form in $ \R^3 $ satisfying $ \inner{\zeta_1}{\zeta_1} = -k^2 $ where $ k = \omega^{1/2} \mu_0 \epsilon_0 $. We are looking for $ w_1 = \sum_0^3 w^l_1 $ with $ w^l_1 \in H^1_\mathrm{loc}(\R^3; \Lambda^l \R^3) $ the solution to \eqref{eq:HL+Q1} of the form
\begin{equation}
w_1 = e_{\zeta_1} (A_{\zeta_1} + R_{\zeta_1}) \label{id:CGO1}
\end{equation}
with $ A_{\zeta_1} $ a constant graded differential form in $ \R^3 $ and $ R_{\zeta_1} \in \dot{X}^{1/2}_{\zeta_1} $. Moreover, we want $ R_{\zeta_1} $ to bear certain sense of smallness. Note that this is equivalent to finding $ R_{\zeta_1} $ that solves
\begin{equation}
(\Delta_{\zeta_1} - k^2) R_{\zeta_1} + Q(\gamma_1, \mu_1, \omega) R_{\zeta_1} = - Q(\gamma_1, \mu_1, \omega) A_{\zeta_1} \label{eq:CGO-HL+Q1}
\end{equation}
in $ \dot{X}^{1/2}_{\zeta_1} $. Note that $Q(\gamma_1,\mu_1,\omega)A_{\zeta_1}\in \dot{X}^{-1/2}_{\zeta_1}$. In the scalar case, this was done in \cite{HT} for such Bourgain-type spaces. In the original case of smooth coefficients, such equation was solved in weighted $L^2$ spaces in \cite{SU} for the scalar case and in \cite{OS} for systems. 

\begin{lemma} \label{lem:CGO-HL+Q1} \sl Let $ \zeta_1 $ and $ A_{\zeta_1} $ be as above.
For $ |\zeta_1| $ large enough, there exists a solution $ R_{\zeta_1} \in \dot{X}^{1/2}_{\zeta_1} $ to \eqref{eq:CGO-HL+Q1} such that
\begin{equation}
\norm{R_{\zeta_1}}{}{\dot{X}^{1/2}_{\zeta_1}} \lesssim \norm{Q(\gamma_1, \mu_1, \omega) A_{\zeta_1}}{}{\dot{X}^{-1/2}_{\zeta_1}}, \label{es:INICIALdecay}
\end{equation}
where the implicit constant (incorporated in the symbol $ \lesssim $) is independent of $ \zeta_1 $.
\end{lemma}
\begin{proof} By using a Neumann series argument (see \cite{SU}), we can show the existence of $ R_{\zeta_1} \in \dot{X}^{1/2}_{\zeta_1} $ satisfying
\[ \norm{R_{\zeta_1}}{}{\dot{X}^{1/2}_{\zeta_1}} \leq \norm{(I + (\Delta_{\zeta_1} - k^2)^{-1}Q(\gamma_1, \mu_1, \omega))^{-1}}{}{\dot{X}^{1/2}_{\zeta_1} \rightarrow \dot{X}^{1/2}_{\zeta_1}} \norm{Q(\gamma_1, \mu_1, \omega) A_{\zeta_1}}{}{\dot{X}^{-1/2}_{\zeta_1}} \]
for $ |\zeta_1| $ large enough, as a simple consequence of \eqref{id:Dz-k2=1} and 
\begin{equation}
\norm{Q(\gamma_1, \mu_1, \omega)}{}{\dot{X}^{1/2}_{\zeta_1} \rightarrow \dot{X}^{-1/2}_{\zeta_1}} = o(\textbf{1}(|\zeta_1|)). \label{es:Q1}
\end{equation}
Here $ \textbf{1}(t) = 1 $ for any $ t \in \R $. 

To prove \eqref{es:Q1}, let $ u $ and $ v $ belong to $ \dot{X}_{\zeta_1}^{1/2} $. By a slight modification of Corollary 2.1 in \cite{HT} we have that
\begin{gather*}
\left| \Duality{Q(\gamma_1, \mu_1, \omega)u}{v} \right| \lesssim |\zeta_1|^{-1} \norm{u}{}{\dot{X}_{\zeta_1}^{1/2}} \norm{v}{}{\dot{X}_{\zeta_1}^{1/2}} \\
+ \left| \int_{\R^3} \inner{\alpha_h}{d \inner{- u^0 + u^2}{v^0 + v^2}} + \inner{\beta_h}{d \inner{u^1 - u^3}{\varphi^1 + v^3}} \, \dd x \right| \\
+ \left| \int_{\R^3} \inner{\beta_h}{D^\ast (u^1 \odot v^1)} \, \dd x \right| + \left| \int_{\R^3} \inner{\alpha_h}{D^\ast (\ast u^2 \odot \ast v^2) } \, \dd x \right| \\
+ \left| \int_{\R^3} \inner{d a_1 - \alpha_h}{d \inner{- u^0 + u^2}{v^0 + v^2}} + \inner{d b_1 - \beta_h}{d \inner{u^1 - u^3}{\varphi^1 + v^3}} \, \dd x \right| \\
+ \left| \int_{\R^3} \inner{d b_1 - \beta_h}{D^\ast (u^1 \odot v^1)} \, \dd x \right| + \left| \int_{\R^3} \inner{d a_1 - \alpha_h}{D^\ast (\ast u^2 \odot \ast v^2) } \, \dd x \right|
\end{gather*}
where $ \alpha_h $ and $ \beta_h $ are $ 1 $-forms in $ \R^3 $ defined by 
\[\alpha_h = \varphi_h \ast d a_1,\qquad \beta_h = \varphi_h \ast d b_1\] 
(here $ \ast $ denotes convolution) with $ 0 < h \leq 1 $, $ \varphi_h (x) = h^{-3} \varphi (x/h) $, $ \varphi \in C^\infty_0(\R^3) $, $ 0 \leq \varphi(x) \leq 1 $ for all $ x \in \R^3 $ and $ \int_{\R^3} \varphi \, \dd x = 1 $. Note that the implicit constant depends on $ \varepsilon_0 $, $ \mu_0 $, $ \Omega $ and the $ C^1 $-norms of $ \gamma_1 $ and $ \mu_1 $. A further modification of Lemma 2.3 in \cite{HT} gives
\begin{gather*}
\left| \Duality{Q(\gamma_1, \mu_1, \omega)u}{v} \right| \lesssim |\zeta_1|^{-1} \norm{u}{}{\dot{X}_{\zeta_1}^{1/2}} \norm{v}{}{\dot{X}_{\zeta_1}^{1/2}} \\
+ |\zeta_1|^{-1} \left( \norm{\delta \alpha_h}{}{L^\infty(\R^3)} + \norm{\delta \beta_h}{}{L^\infty(\R^3)} \right) \norm{u}{}{\dot{X}_{\zeta_1}^{1/2}} \norm{v}{}{\dot{X}_{\zeta_1}^{1/2}} \\
+ \left( \norm{d a_1 - \alpha_h}{}{L^\infty(\R^3)} + \norm{d b_1 - \beta_h}{}{L^\infty(\R^3)} \right) \norm{u}{}{\dot{X}_{\zeta_1}^{1/2}} \norm{v}{}{\dot{X}_{\zeta_1}^{1/2}} \\
\lesssim \left( |\zeta_1|^{-1} h^{-1} + o(\textbf{1}(h)) \right) \norm{u}{}{\dot{X}_{\zeta_1}^{1/2}} \norm{v}{}{\dot{X}_{\zeta_1}^{1/2}}
\end{gather*}
as $ h $ vanishes. Choosing $ h = |\zeta_1|^{-1/2} $, this implies \eqref{id:Dz-k2=1} and the lemma is proven.
\end{proof}

Up to this point, nothing has been said about the smallness of $ R_{\zeta_1} $. We will see in the next lemma that estimate \eqref{es:INICIALdecay} yields such smallness in an average sense. This idea, due to Haberman and Tataru, is one of the key points in \cite{HT}.

\begin{lemma} \label{lem:ave-decay} \sl Let $s\in\R$ satisfy $s\geq1$. Given a real-valued constant $ 1 $-form $ \rho $ in $\R^3$, choose $ \eta_1 $ and $ \eta_2 $ also real-valued constant $ 1 $-forms such that $ \inner{\eta_1}{\eta_2} = 0 $, $ \inner{\eta_j}{\rho} = 0 $ and $ |\eta_j| = 1 $ for $ j \in \{ 1, 2 \} $. Set
\[ \zeta_1 = -\sqrt{s^2+\frac{|\rho|^2}{4}}~\eta_1+i\left(\frac{\rho}{2}-\sqrt{s^2+k^2}~\eta_2\right), \]
and assume $ |A_{\zeta_1}| $ is bounded as a function of $ s, \eta_1 $. Then the $ R_{\zeta_1} $ obtained in Lemma \ref{lem:CGO-HL+Q1} satisfies
\begin{equation}\label{eq:avg-decay-multi-Q}
\frac{1}{\lambda}\int_{S^1}\int_\lambda^{2\lambda}\norm{R_{\zeta_1}}{2}{\dot{X}^{1/2}_{\zeta_1}}~\dd s~\dd\eta_1 = o(\textbf{1}(\lambda))
\end{equation}
as $ \lambda $ becomes large. Here $ S^1 $ denotes the intersection between the unit sphere in $ \R^3 $ and the plane defined by $ \eta_1 $ and $ \eta_2 $.
\end{lemma}
\begin{proof} By the definition of $ Q(\gamma_1, \mu_1, \omega) $, the identity \eqref{eq:lastterm}, the fact that $ A_{\zeta_1} $ is constant and that $ Q(\gamma_1, \mu_1, \omega) $ is compactly supported, we have
\begin{equation*}
\left| \Duality{Q(\gamma_1, \mu_1, \omega)A_{\zeta_1}}{v} \right| \lesssim \sum_{l = 0}^3 \norm{\chi v^l}{}{L^2(\R^3; \Lambda^l \R^3)} + \norm{\chi (d + \delta) f_{\zeta_1}}{}{\dot{X}^{-1/2}_{\zeta_1}} \norm{v}{}{\dot{X}^{1/2}_{\zeta_1}}
\end{equation*}
where $ v = \sum_0^3 v^l $, $ \chi \in C^\infty_0 (\R^3) $ such that $ \chi (x) = 1 $ for all $ x \in \supp d \gamma_1 \cup \supp d \mu_1 $ and
\[f_{\zeta_1} = db_1\wedge A^1_{\zeta_1} + db_1\vee(A^1_{\zeta_1}+A^3_{\zeta_1})+da_1\wedge(A^0_{\zeta_1}+A^2_{\zeta_1})-da_1\vee A^2_{\zeta_1}\]
with $ A_{\zeta_1} = \sum_0^3 A_{\zeta_1}^l $.
By (5) in Lemma 2.2 of \cite{HT}, this gives that
\[ \norm{Q(\gamma_1, \mu_1, \omega)A_{\zeta_1}}{}{\dot{X}^{-1/2}_{\zeta_1}} \lesssim s^{-1/2} + \norm{\chi (d + \delta) f_{\zeta_1}}{}{\dot{X}^{-1/2}_{\zeta_1}}. \]
Now, an immediate modification of Lemma 3.1 in \cite{HT} allows to check that
\[ \frac{1}{\lambda}\int_{S^1}\int_\lambda^{2\lambda}\norm{\chi (d + \delta) f_{\zeta_1}}{2}{\dot{X}^{-1/2}_{\zeta_1}}~\dd s~\dd\eta_1 = o(\textbf{1}(\lambda)), \]
which implies 
\begin{equation}
\frac{1}{\lambda}\int_{S^1}\int_\lambda^{2\lambda}\norm{Q(\gamma_1, \mu_1, \omega) A_{\zeta_1}}{2}{\dot{X}^{-1/2}_{\zeta_1}}~\dd s~\dd\eta_1 = o(\textbf{1}(\lambda))\label{es:deacyQA}
\end{equation}
as $\lambda$ becomes large. By \eqref{es:INICIALdecay}, we obtain \eqref{eq:avg-decay-multi-Q}.
\end{proof}

From the construction of $ R_{\zeta_1} \in \dot{X}^{1/2}_{\zeta_1} $ solving \eqref{eq:CGO-HL+Q1}, the existence of $ w_1 $ of the form \eqref{id:CGO1} that solves \eqref{eq:HL+Q1} is immediate. However, it turns out that for such $ w_1 $ to satisfy the condition in Proposition \ref{prop:int-formula}, the constant 1-form $A_{\zeta_1}$ has to be chosen carefully.
\begin{lemma} \sl Let $w_1 = \sum_{l=0}^3w_1^l$ as in \eqref{id:CGO1} with $ \zeta_1 $, $ A_{\zeta_1} $ and  $ R_{\zeta_1} $ as in Lemma \ref{lem:CGO-HL+Q1}. Then $ w_1^l \in H^1_\mathrm{loc} (\R^3; \Lambda^l \R^3) $ and $ w_1 $ is a solution of \eqref{eq:HL+Q1}. Moreover, if $ A_{\zeta_1} $ satisfies the relation
\begin{equation}\label{eq:cond-incid}
-{\zeta_1}\vee A_{\zeta_1}^1+ikA_{\zeta_1}^0-{\zeta_1}\wedge A_{\zeta_1}^2+ikA_{\zeta_1}^3=0,
\end{equation}
then $ v_1 = \sum_0^3 v_1^l $ defined as in \eqref{eq:v1ofw1} satisfies $ v_1^0 + v_1^3 = 0 $ for $ |\zeta_1| $ large enough.
\end{lemma}
\begin{proof} We can ensure $ w_1^l $ is in $ H^1_\mathrm{loc} (\R^3; \Lambda^l \R^3) $ since $ R_{\zeta_1} \in \dot{X}^{1/2}_{\zeta_1} $ (See Remark \ref{rem:smoothness}). Additionally, $ w_1 $ is a solution of \eqref{eq:HL+Q1} since $ R_{\zeta_1} \in \dot{X}^{1/2}_{\zeta_1} $ solves\footnote{See also \eqref{eq:DELTA-zeta}.} \eqref{eq:CGO-HL+Q1}. 

In order to prove the second part of this lemma, note that $ v_1^l \in H^1_\mathrm{loc} (\R^3; \Lambda^l \R^3) $ and
\[P(d + \delta; \gamma_1, \mu_1, \omega) v_1 = 0\] 
in any bounded open subset of $ \R^3 $ by Proposition \ref{prop:HL+Q}. Then by Proposition \ref{prop:v0v3eq0} we know that $ v^0_1 + v_1^3 $ is a weak solution to
\[ (\delta d + d \delta - k^2) (v^0_1 + v_1^3) + \tilde{q}(\gamma_1, \mu_1, \omega) (v^0_1 + v_1^3) = 0 \]
in $ \R^3 $. By \eqref{eq:v1ofw1}, we can write $ v_1^l = e_{\zeta_1} (B^l_{\zeta_1} + S^l_{\zeta_1}) $ with $ l \in \{ 0, 3 \} $ where
\begin{align}
B^0_{\zeta_1} &= -{\zeta_1}\vee A_{\zeta_1}^1+ikA_{\zeta_1}^0, \nonumber \\
S^0_{\zeta_1} &= -{\zeta_1}\vee R_{\zeta_1}^1 + \delta R_{\zeta_1}^1 + db \vee (A_{\zeta_1}^1 + R_{\zeta_1}^1) + i \omega \gamma_1^{1/2} \mu_1^{1/2} R_{\zeta_1}^0 \label{id:S0zeta1} \\
& + i (\omega \gamma_1^{1/2} \mu_1^{1/2} - k) A_{\zeta_1}^0, \nonumber \\
B^3_{\zeta_1} &= -{\zeta_1}\wedge A_{\zeta_1}^2+ikA_{\zeta_1}^3, \nonumber \\
S^3_{\zeta_1} &= -{\zeta_1}\wedge R_{\zeta_1}^2 - d R_{\zeta_1}^2 + da \wedge (A_{\zeta_1}^2 + R_{\zeta_1}^2) + i \omega \gamma_1^{1/2} \mu_1^{1/2} R_{\zeta_1}^3 \label{id:S3zeta1} \\
& + i (\omega \gamma_1^{1/2} \mu_1^{1/2} - k) A_{\zeta_1}^3. \nonumber
\end{align}
Then relation \eqref{eq:cond-incid} implies $B_{\zeta_1}^0+B_{\zeta_1}^3=0$ hence that $ v^0_1 + v_1^3 = e_{\zeta_1} (S^0_{\zeta_1} + S^3_{\zeta_1}) $ is a weak solution of
\begin{equation}
(\Delta_{\zeta_1} - k^2) (S^0_{\zeta_1} + S^3_{\zeta_1}) + \tilde{q}(\gamma_1, \mu_1, \omega) (S^0_{\zeta_1} + S^3_{\zeta_1}) = 0 \label{eq:CGOv0v3}
\end{equation}
in $ \R^3 $. 

To complete the proof, it is sufficient to show that \eqref{eq:CGOv0v3} is uniquely solvable in $ \dot{X}^{1/2}_{\zeta_1} $ for $ |\zeta_1| $ large enough and $ S^0_{\zeta_1} + S^3_{\zeta_1} $ belongs to $ \dot{X}^{1/2}_{\zeta_1} $.

Using the same argument as in proving \eqref{es:Q1}, we see that $ \tilde{q}(\gamma_1, \mu_1, \omega) $ is a bounded linear operator from $ \dot{X}^{1/2}_{\zeta_1} $ to $ \dot{X}^{-1/2}_{\zeta_1} $ and its operator norm is $ o(\mathbf{1}(|\zeta_1|)) $. Then, by Remark \ref{rem:uniqueness}, identity \eqref{id:Dz-k2=1} and the Banach fixed-point theorem, equation \eqref{eq:CGOv0v3} is unique solvable in $ \dot{X}^{1/2}_{\zeta_1} $ for $|\zeta_1|$ large enough.

Since $ e_{\zeta_1} S_{\zeta_1}^l = v^l_1 \in H^1_\mathrm{loc} (\R^3; \Lambda^l \R^3) $ for $ l \in \{ 0, 3 \} $, we know that $ \chi (S^0_{\zeta_1} + S^3_{\zeta_1}) \in \dot{X}^{1/2}_{\zeta_1} $ for $ \chi \in C^\infty_0 (\R^3) $ such that $ \chi (x) = 1 $ for all $ x \in (\supp d \gamma_1 \cup \supp d \mu_1) $. Therefore, the right hand side of
\[ (\Delta_{\zeta_1} - k^2) (S^0_{\zeta_1} + S^3_{\zeta_1}) = - \tilde{q}(\gamma_1, \mu_1, \omega) \chi (S^0_{\zeta_1} + S^3_{\zeta_1}) \]
is in $\dot{X}^{-1/2}_{\zeta_1}$. Further, it is not hard to see from \eqref{id:S0zeta1} and \eqref{id:S3zeta1} that $ \widehat{S^l_{\zeta_1}} $ belongs to $ L^2_\mathrm{loc} (\R^3; \Lambda^l \R^3) $ with $ l \in \{ 0, 3 \} $. The last two facts imply that $ S^0_{\zeta_1} + S^3_{\zeta_1} \in \dot{X}^{1/2}_{\zeta_1} $.
\end{proof}

\begin{remark} \sl The condition given by \eqref{eq:cond-incid} is necessary in our proof since $ B^0_{\zeta_1} + B^3_{\zeta_1} $ does not belong to $ \dot{X}^{1/2}_{\zeta_1} $.
\end{remark}

As a conclusion of these lemmas, we can state the constructions of $ w_1 $ in the following theorem.

\begin{theorem} \label{th:CGOw_1} \sl Let $ s \in \R $ satisfy $ s \geq 1 $. Given a real-valued constant $ 1 $-form $ \rho $ in $ \R^3 $, choose $ \eta_1 $ and $ \eta_2 $ also real-valued constant $ 1 $-forms in $ \R^3 $ such that $ \inner{\eta_1}{\eta_2} = 0 $, $ \inner{\eta_j}{\rho} = 0 $ and $ |\eta_j| = 1 $ for $ j \in \{ 1, 2 \} $. Set
\[ \zeta_1 = -\sqrt{s^2+\frac{|\rho|^2}{4}}~\eta_1+i\left(\frac{\rho}{2}-\sqrt{s^2+k^2}~\eta_2\right) \]
and
\[ A_{\zeta_1} = \frac{\sqrt 2}{|\zeta_1|} \left( \zeta_1 \vee \alpha + ik \alpha + ik \beta + \zeta_1 \wedge \beta \right), \]
where either $ \alpha = \eta_1 $ and $ \beta = 0 $ or $ \alpha = 0 $ and $ \beta = |\rho|^{-1} \eta_2 \wedge \rho $. Then for $ |\zeta_1| $ large enough, there exists $ w_1 = \sum_0^3 w_1^l $ with $ w_1^l \in H^1_\mathrm{loc} (\R^3; \Lambda^l \R^3) $ of the form
\[ w_1 = e_{\zeta_1} (A_{\zeta_1} + R_{\zeta_1}),\]
which is a weak solution to
\[ (d \delta + \delta d - k^2)w_1 + Q(\gamma_1, \mu_1, \omega) w_1 = 0 \]
in $ \R^3 $. Moreover, we have $ R_{\zeta_1} \in \dot{X}^{1/2}_{\zeta_1} $ satisfies
\[ \frac{1}{\lambda}\int_{S^1}\int_\lambda^{2\lambda}\norm{R_{\zeta_1}}{2}{\dot{X}^{1/2}_{\zeta_1}}~\dd s~\dd\eta_1 = o(\textbf{1}(\lambda)) \]
as $ \lambda $ becomes large.  Here $ S^1 $ denotes the intersection between the unit sphere in $ \R^3 $ and the plane defined by $ \eta_1 $ and $ \eta_2 $. Furthermore, $ v_1 = \sum_0^3 v_1^l $ defined by
\begin{equation*}
v_1 = P(d + \delta ;\gamma_1, \mu_1, \omega)^t w_1
\end{equation*}
satisfies $ v_1^0 + v_1^3 = 0 $ for $ |\zeta_1| $ large enough.
\end{theorem}

\subsection{The construction of $ v_2 $.} Let $ \zeta_2 $ be a complex-valued constant $ 1 $-form in $ \R^3 $ satisfying $ \inner{\zeta_2}{\zeta_2} = -k^2 $. We are looking for the solution $ v_2 = \sum_0^3 v^l_2 $ with $ v^l_2 \in H^1_\mathrm{loc}(\R^3; \Lambda^l \R^3) $ to \eqref{eq:Ptv2=0} in any bounded subset of $ \R^3 $ of the form
\begin{equation}
v_2 = e_{\zeta_2} (B_{\zeta_2} + S_{\zeta_2}) \label{id:CGO2}
\end{equation}
where $ B_{\zeta_2} $ is a constant graded differential form in $ \R^3 $ and $ S_{\zeta_2} \in \dot{X}^{1/2}_{\zeta_2} $. In addition, we want $ S_{\zeta_2} $ to be small in the sense of \eqref{eq:avg-decay-multi-Q}. To construct such $ v_2 $, by Proposition \ref{prop:HL+tQ}, we start with the construction of a solution $ w_2 $ to 
\begin{equation}\label{eq:HlmtQw2}
\int_{\R^3} \inner{\delta w_2}{\delta \varphi} + \inner{d w_2}{d \varphi} - \omega^2 \varepsilon_0 \mu_0 \inner{w_2}{\varphi} \, \dd x + \Duality{\tilde{Q} (\gamma_2, \mu_2, \omega) w_2}{\varphi} = 0 
\end{equation}
for all $ \varphi = \sum_0^3 \varphi^l $ with $ \varphi^l \in C^\infty_0 (\R^3;\Lambda^l \R^3) $.

\begin{lemma} \label{lem:HL+Q2} \sl Let $ A_{\zeta_2} = A_{\zeta_2}^1 + A_{\zeta_2}^2 $ be a constant graded differential form in $ \R^3 $. For $ |\zeta_2| $ large enough, there exists $ R_{\zeta_2} = R_{\zeta_2}^1 + R_{\zeta_2}^2 \in \dot{X}^{1/2}_{\zeta_2} $ such that $ w_2 = w_2^1 + w_2^2 $ with
\[ w_2^l = e_{\zeta_2} (A_{\zeta_2}^l + R_{\zeta_2}^l) \]
and $ w_2^l \in H^1_\mathrm{loc} (\R^3; \Lambda^l \R^3) $, is a solution of \eqref{eq:HlmtQw2} in $ \R^3 $.
\end{lemma}
\begin{proof} Analogous to the proof of Lemma \ref{lem:CGO-HL+Q1}, the existence of a general $R_{\zeta_2}=\sum_{0}^3R_{\zeta_2}^l$ for given constant $A_{\zeta_2}=\sum_{0}^3A_{\zeta_2}^l$ is immediate by
\[ \norm{\tilde{Q}(\gamma_2, \mu_2, \omega)}{}{\dot{X}^{1/2}_{\zeta_2} \rightarrow \dot{X}^{-1/2}_{\zeta_2}} = o(\textbf{1}(|\zeta_2|)) \]
as $ |\zeta_2| $ becomes large.  Since $ \tilde{Q}(\gamma_2, \mu_2, \omega) $ decouples for $ 1 $ and $ 2 $ forms, therefore we can ensure that $ R_{\zeta_2} = R_{\zeta_2}^1 + R_{\zeta_2}^2 $ for $A_{\zeta_2}=A_{\zeta_2}^1+A_{\zeta_2}^2$. 
\end{proof}

Now Proposition \ref{prop:HL+tQ} states that $v_2= P(d + \delta; \gamma_2, \mu_2, \omega) w_2$ is then a solution to \eqref{eq:Ptv2=0}. Moreover, we can write $v_2$ as in\eqref{id:CGO2}. However, we need still to show the smallness of $ S_{\zeta_2} $.

\begin{theorem} \label{th:CGOv_2} \sl Let $ s \in \R $ satisfy $ s \geq 1 $. Given a real-valued constant $ 1 $-form $ \rho $ in $ \R^3 $, we choose $ \eta_1 $ and $ \eta_2 $ two other real-valued constant $ 1 $-forms in $ \R^3 $ such that $ \inner{\eta_1}{\eta_2} = 0 $, $ \inner{\eta_j}{\rho} = 0 $ and $ |\eta_j| = 1 $ for $ j \in \{ 1, 2 \} $. Set
\[ \zeta_2 = \sqrt{s^2+\frac{|\rho|^2}{4}}~\eta_1+i\left(\frac{\rho}{2}+\sqrt{s^2+k^2}~\eta_2\right) \]
and $ \alpha $ and $ \beta $ be as in Theorem \ref{th:CGOw_1}. If $ |\zeta_2| $ is large enough, there exists $ v_2 = \sum_0^3 v^l_2 $ with $ v^l_2 \in H^1_\mathrm{loc} (\R^3; \Lambda^l \R^3) $ of the form 
\[ v_2 = e_{\zeta_2} (B_{\zeta_2} + S_{\zeta_2}), \]
where
\begin{equation}\label{eq:B-zeta2}
B_{\zeta_2} = - \frac{\sqrt{2}}{|\zeta_2|} (\zeta_2 \vee (\alpha + \beta) + \zeta_2 \wedge (- \alpha + \beta) + ik (\alpha + \beta)) 
\end{equation}
and $ S_{\zeta_2} \in \dot{X}^{1/2}_{\zeta_2} $, which solves 
\[ P(d + \delta; \gamma_2, \mu_2, \omega)^t v_2 = 0 \]
in any bounded open subset of $ \R^3 $ and satisfies
\begin{equation}\label{eq:avgdecay-S-zeta2}
\frac{1}{\lambda}\int_{S^1}\int_\lambda^{2\lambda}\norm{S_{\zeta_2}}{2}{\dot{X}^{1/2}_{\zeta_2}}~\dd s~\dd\eta_1 = o(\textbf{1}(\lambda))
\end{equation}
as $ \lambda $ becomes large. Here $ S^1 $ denotes the intersection between the unit sphere in $ \R^3 $ and the plane defined by $ \eta_1 $ and $ \eta_2 $.
\end{theorem}
\begin{proof} Let $ w_2 $ be as in Lemma \ref{lem:HL+Q2} with $ A_{\zeta_2} = A_{\zeta_2}^1 + A_{\zeta_2}^2 = - \sqrt{2} (\alpha + \beta) $. By Proposition \ref{prop:HL+tQ}, we know that $ v_2 = \sum_0^3 v^l_2 $ defined by
\[ v_2 = P(d + \delta; \gamma_2, \mu_2, \omega) w_2 \]
satisfies that $ v^l_2 \in H^1_\mathrm{loc} (\R^3; \Lambda^l \R^3) $ and solves
\begin{equation}
P(d + \delta; \gamma_2, \mu_2, \omega)^t v_2 = 0 \label{eq:Ptv2}
\end{equation}
in any bounded open subset of $ \R^3 $. One can easily write
\[ v_2 = e_{\zeta_2} (B_{\zeta_2} + S_{\zeta_2}) \]
and check that $ B_{\zeta_2} $ is given by \eqref{eq:B-zeta2} and
\begin{align*}
S_{\zeta_2} &= \frac{1}{|\zeta_2|} (\zeta_2 \vee (R_{\zeta_2}^1 + R_{\zeta_2}^2) + \zeta_2 \wedge (- R_{\zeta_2}^1 + R_{\zeta_2}^2) + (d + \delta) (- R_{\zeta_2}^1 + R_{\zeta_2}^2) \\
& + da_2 \wedge (A_{\zeta_2}^1 + R_{\zeta_2}^1) + da_2 \vee (A_{\zeta_2}^1 + R_{\zeta_2}^1) + db_2 \wedge (A_{\zeta_2}^2 + R_{\zeta_2}^2)\\
& - db_2 \vee (A_{\zeta_2}^2 + R_{\zeta_2}^2) + i \omega \gamma_2^{1/2} \mu_2^{1/2} (R_{\zeta_2}^1 + R_{\zeta_2}^2) + i (\omega \gamma_2^{1/2} \mu_2^{1/2} - k)(A_{\zeta_2}^1 + A_{\zeta_2}^2)).
\end{align*}
Moreover, by \eqref{eq:Ptv2} and \eqref{id:giving_HL+Q}, we know that $ S_{\zeta_2} $ satisfies the familiar equation
\begin{equation}
(\Delta_{\zeta_2} - k^2) S_{\zeta_2} + Q(\gamma_2, \mu_2, \omega) S_{\zeta_2} = - Q(\gamma_2, \mu_2, \omega) B_{\zeta_2}. \label{eq:CGOHL+Q2}
\end{equation}

Since $ Q(\gamma_2, \mu_2, \omega) B_{\zeta_2} \in \dot{X}^{-1/2}_{\zeta_2} $, equation \eqref{eq:CGOHL+Q2} is uniquely solvable in $ \dot{X}^{1/2}_{\zeta_2} $. Therefore, since $ S_{\zeta_2} \in \dot{X}^{1/2}_{\zeta_2} $ and $|B_{\zeta_2}|=\mathcal O\left(\textbf{1}(|\zeta_2|)\right)$, $S_{\zeta_2}$ satisfies \eqref{eq:avgdecay-S-zeta2}. 
\end{proof}

\section{Proof of uniqueness}\label{sec:uniqueness}
To complete the proof of Theorem \ref{th:main}, the final step is to plug into the integral formula given in Proposition \ref{prop:int-formula} the $ w_1 $ and $ v_2 $ obtained in Theorem \ref{th:CGOw_1} and Theorem \ref{th:CGOv_2} and to let $\lambda$ go to $\infty$. The output turns out to be certain non-linear relations of $ \gamma_1 $, $ \mu_1 $, $ \gamma_2 $, $ \mu_2 $ and their weak partial derivatives up to the second order. Then a unique continuation principle argument can be used to conclude the uniqueness.

Throughout this section we let $ Q_j $ denote $ Q(\gamma_j, \mu_j, \omega) $ with $ j \in \{ 1, 2 \} $. If
\begin{align*}
A_1 &= - (\eta_1 + i \eta_2) \vee \alpha - (\eta_1 + i \eta_2) \wedge \beta,\\ 
B_2 &= - (\eta_1 + i \eta_2) \vee (\alpha + \beta) - (\eta_1 + i \eta_2) \wedge (- \alpha + \beta)
\end{align*}
with $ \alpha $ and $ \beta $ as in Theorem \ref{th:CGOw_1}, we see that, for any $ \rho $, $ |A_{\zeta_1} - A_1| + |B_{\zeta_2} - B_2| = \mathcal{O}(s^{-1}) $ for $ s $ large enough and all $ \eta_1, \eta_2 \in S^1 $. The implicit constant (incorporated in the symbol $\mathcal O$) here depends on $ \rho $. On the other hand, plugging $ w_1 $ and $ v_2 $, as in Theorem \ref{th:CGOw_1} and Theorem \ref{th:CGOv_2}, into Proposition \ref{prop:int-formula} we get
\begin{gather*}
\Duality{(Q_2 - Q_1) e_{i\rho} A_1}{B_2} = \Duality{(Q_1 - Q_2) (A_{\zeta_1} + R_{\zeta_1})}{e_{i\rho} (B_{\zeta_2} - B_2 + S_{\zeta_2})} \\
+ \Duality{(Q_1 - Q_2) B_2 }{e_{i\rho} (A_{\zeta_1} - A_1 + R_{\zeta_1})}.
\end{gather*}
We know that, for each $ \rho $, $ Q_j $ is bounded from $ \dot{X}^{1/2}_{\zeta_j} $ to $ \dot{X}^{-1/2}_{\zeta_j} $ and its norm is $ o(\textbf{1}(s)) $ for $ s $ large enough and all $ \eta_1 $ (see \eqref{es:Q1} and the same applies to $Q_2$). The same is true for $ Q_1 - Q_2 $ from $ \dot{X}^{1/2}_{\zeta_1} $ to $ \dot{X}^{-1/2}_{\zeta_2} $ as an immediate consequence of the proof of Lemma 2.3 in \cite{HT}. Thus, for each $ \rho $ we have
\begin{gather}\label{eq:Q2-Q1-A1-B2}
\left| \Duality{(Q_2 - Q_1) e_{i\rho} A_1}{B_2} \right| \lesssim 
\norm{(Q_1 - Q_2) B_2}{}{\dot{X}^{-1/2}_{\zeta_1}} \left[ \norm{\chi (A_{\zeta_1} - A_1)}{}{\dot{X}^{1/2}_{\zeta_1}} + \norm{R_{\zeta_1}}{}{\dot{X}^{1/2}_{\zeta_1}} \right] \\
+ \left[ \norm{(Q_1 - Q_2) A_{\zeta_1}}{}{\dot{X}^{-1/2}_{\zeta_2}} + \norm{R_{\zeta_1}}{}{\dot{X}^{1/2}_{\zeta_1}} \right] \left[ \norm{\chi (B_{\zeta_2} - B_2)}{}{\dot{X}^{1/2}_{\zeta_2}} + \norm{S_{\zeta_2}}{}{\dot{X}^{1/2}_{\zeta_2}} \right] \nonumber
\end{gather}
where $ \chi \in \C^\infty_0 (\R^3) $ such that $ \chi(x) = 1 $ for all $ x \in \supp d \gamma_2 \cup \supp d \mu_2 $. Here the implicit constant might depends on $ \rho $. 

If $ \alpha = \eta_1 $ and $ \beta = 0 $, then $ A_1 = -1 $, $ B_2 = -1 + i \eta_2 \wedge \eta_1 $ and the left hand side of \eqref{eq:Q2-Q1-A1-B2} gives
\begin{equation}\label{eq:Q2-Q1-A1-B2-1}
\begin{aligned}
&\Duality{(Q_2 - Q_1) e_{i\rho} A_1}{B_2} = \int_{\R^3} \inner{d (a_1 - a_2)}{d e_{i \rho}} \, \dd x \\
&+ \int_{\R^3} \inner{d (a_1 + a_2)}{d (a_2 - a_1)} e_{i \rho} \, \dd x + \int_{\R^3} \omega^2 (\gamma_1 \mu_1 - \gamma_2 \mu_2) e_{i \rho} \, \dd x.
\end{aligned}
\end{equation}
If $ \alpha = 0 $ and $ \beta = |\rho|^{-1} \eta_2 \wedge \rho $, then $ A_1 = - |\rho|^{-1} \eta_1 \wedge \eta_2 \wedge \rho $, $ B_2 = - |\rho|^{-1} (\eta_1 + i \eta_2) \vee (\eta_2 \wedge \rho)  - |\rho|^{-1} \eta_1 \wedge \eta_2 \wedge \rho $ and we have
\begin{equation}\label{eq:Q2-Q1-A1-B2-2}
\begin{aligned}
&\Duality{(Q_2 - Q_1) e_{i\rho} A_1}{B_2} = \int_{\R^3} \inner{d (b_1 - b_2)}{d e_{i \rho}} \, \dd x \\
&+ \int_{\R^3} \inner{d (b_1 + b_2)}{d (b_2 - b_1)} e_{i \rho} \, \dd x + \int_{\R^3} \omega^2 (\gamma_1 \mu_1 - \gamma_2 \mu_2) e_{i \rho} \, \dd x.
\end{aligned}\end{equation}
Meanwhile, by the choice of $A_1$ and $B_2$ above, we have
\begin{align*}
\left( \frac{1}{\lambda}\int_{S^1}\int_\lambda^{2\lambda}\norm{\chi (A_{\zeta_1} - A_1)}{2}{\dot{X}^{1/2}_{\zeta_1}}~\dd s~\dd\eta_1 \right)^{1/2} = \mathcal{O} (\textbf{1}(\lambda)), \\
\left( \frac{1}{\lambda}\int_{S^1}\int_\lambda^{2\lambda}\norm{\chi (B_{\zeta_2} - B_2)}{2}{\dot{X}^{1/2}_{\zeta_2}}~\dd s~\dd\eta_1 \right)^{1/2} = \mathcal{O} (\textbf{1}(\lambda)).
\end{align*}
Then after averaging \eqref{eq:Q2-Q1-A1-B2} on $(s, \eta_1)\in[\lambda, 2\lambda]\times S^1$ and using the Cauchy-Schwartz inequality, we get
\begin{gather*}
\left| \Duality{(Q_2 - Q_1) e_{i\rho} A_1}{B_2} \right| \\
\lesssim [ \mathcal{O} (\textbf{1}(\lambda)) + o (\textbf{1}(\lambda)) ] \left( \frac{1}{\lambda}\int_{S^1}\int_\lambda^{2\lambda} \norm{(Q_1 - Q_2) B_2}{2}{\dot{X}^{-1/2}_{\zeta_1}} ~\dd s~\dd\eta_1 \right)^{1/2} \\
+ [ \mathcal{O} (\textbf{1}(\lambda)) + o (\textbf{1}(\lambda)) ] \left[ \left( \frac{1}{\lambda}\int_{S^1}\int_\lambda^{2\lambda} \norm{(Q_1 - Q_2) A_{\zeta_1}}{2}{\dot{X}^{-1/2}_{\zeta_2}} ~\dd s~\dd\eta_1 \right)^{1/2} + o (\textbf{1}(\lambda)) \right]
\end{gather*}
where Theorem \ref{th:CGOw_1} and Theorem \ref{th:CGOv_2} are used. It is not hard to see this converges to zero as $\lambda$ goes to $\infty$ by the same argument we used in proving \eqref{es:deacyQA} and noticing that the left hand side is independent of $\lambda$. Thus, by \eqref{eq:Q2-Q1-A1-B2-1} and \eqref{eq:Q2-Q1-A1-B2-2} we arrive at
\begin{equation}
\begin{gathered}
\int_{\R^3} \inner{d (a_2 - a_1)}{d e_{i \rho}} \, \dd x - \int_{\R^3} \inner{d (a_1 + a_2)}{d (a_2 - a_1)} e_{i \rho} \, \dd x \\
+ \int_{\R^3} \omega^2 (\gamma_2 \mu_2 - \gamma_1 \mu_1) e_{i \rho} \, \dd x = 0
\end{gathered} \label{eq:a2-a1}
\end{equation}
and
\begin{equation}
\begin{gathered}
\int_{\R^3} \inner{d (b_2 - b_1)}{d e_{i \rho}} \, \dd x - \int_{\R^3} \inner{d (b_1 + b_2)}{d (b_2 - b_1)} e_{i \rho} \, \dd x \\
+ \int_{\R^3} \omega^2 (\gamma_2 \mu_2 - \gamma_1 \mu_1) e_{i \rho} \, \dd x = 0
\end{gathered} \label{eq:b2-b1}
\end{equation}
for any $ \rho $. So far, this shows that
\[\left\{\begin{array}{l}
\delta d (a_2 - a_1) - \inner{d (a_1 + a_2)}{d (a_2 - a_1)} + \omega^2(\gamma_2 \mu_2 - \gamma_1 \mu_1) = 0, \\
\delta d (b_2 - b_1) - \inner{d (b_1 + b_2)}{d (b_2 - b_1)} + \omega^2(\gamma_2 \mu_2 - \gamma_1 \mu_1) = 0,
\end{array}\right.
\]
a system that has to be understood in the weak sense. 
Finally, some simple computations yield a system of second order equations of the form
\[\left\{\begin{array}{l}
- \Delta (\gamma_2^{1/2} - \gamma_1^{1/2}) + V (\gamma_2^{1/2} - \gamma_1^{1/2}) + a (\gamma_2^{1/2} - \gamma_1^{1/2}) + b (\mu_2^{1/2} - \mu_1^{1/2}) = 0\\
- \Delta (\mu_2^{1/2} - \mu_1^{1/2}) + W (\mu_2^{1/2} - \mu_1^{1/2}) + c (\mu_2^{1/2} - \mu_1^{1/2}) + d (\gamma_2^{1/2} - \gamma_1^{1/2}) = 0,
\end{array}\right.
\]
again in the weak sense with 
\begin{align*}
V = - \frac{\delta d (\gamma_1^{1/2} + \gamma_2^{1/2})}{\gamma_1^{1/2} + \gamma_2^{1/2}},\qquad W = - \frac{\delta d (\mu_1^{1/2} + \mu_2^{1/2})}{\mu_1^{1/2} + \mu_2^{1/2}}
\end{align*}
and
\begin{align*}
a = \mathbf{1}_\Omega \omega^2 \gamma_1^{1/2} \gamma_2^{1/2} (\mu_1 + \mu_2), & & b = - \mathbf{1}_\Omega \omega^2 \gamma_1^{1/2} \gamma_2^{1/2} (\gamma_1 + \gamma_2) \frac{\mu_1^{1/2} + \mu_2^{1/2}}{\gamma_1^{1/2} + \gamma_2^{1/2}}, \\
c = \mathbf{1}_\Omega \omega^2 \mu_1^{1/2} \mu_2^{1/2} (\gamma_1 + \gamma_2), & & d = - \mathbf{1}_\Omega \omega^2 \mu_1^{1/2} \mu_2^{1/2} (\mu_1 + \mu_2) \frac{\gamma_1^{1/2} + \gamma_2^{1/2}}{\mu_1^{1/2} + \mu_2^{1/2}},
\end{align*}
where $ \mathbf{1}_\Omega $ is the characteristic function of $ \Omega $. Note that $ \gamma_2^{1/2} - \gamma_1^{1/2} $ and $ \mu_2^{1/2} - \mu_1^{1/2} $ belong to $ H^1 (\R^3) $ and they are compactly supported. Thus, the next unique continuation result implies that $ \gamma_2 = \gamma_1 $ and $ \mu_2 = \mu_1 $.
\begin{lemma} \sl Let $ f $ and $ g $ belong to $ H^1 (\R^3) $ and assume that they are compactly supported. Then, $ f $ and $ g $ vanish if and only if they satisfy
\begin{equation}\left\{
\begin{array}{l}
- \Delta f + V f + a f + b g = 0\\
- \Delta g + W g + c g + d f = 0.
\end{array}\right.
\label{linear_sys}
\end{equation}
\end{lemma}
\begin{proof} Let $ \zeta \in \C^n $ satisfies $ \zeta \cdot \zeta = 0 $. Set $ u(x) = e^{\zeta \cdot x} f(x) $ and $ v(x) = e^{\zeta \cdot x} g(x) $. Since $ f $ and $ g $ belong to $ H^1(\R^3) $ and they are compactly supported, then $ u $ and $ v $ also belong to $ H^1 (\R^3) $ and, consequently, to $ \dot{X}^{1/2}_\zeta $. Moreover, $ u $ and $ v $ solve
\begin{equation}\label{eq:final-u-v}
\left\{
\begin{array}{l}
-( \Delta + 2 \zeta \cdot \nabla) u + V u + a u + b v = 0,\\
-( \Delta + 2 \zeta \cdot \nabla) v + W v + c v + d u = 0.
\end{array}\right.
\end{equation}
Let $ w = w^0 + w^3 $ be the graded form given by $ w^0 = u $ and $ w^3 = \ast v $ and define
\begin{gather*}
\Duality{Q w}{\varphi} = - \int_{\R^3} \inner{d (\gamma_1^{1/2} + \gamma_2^{1/2})}{d \frac{\inner{w^0}{\varphi^0}}{\gamma_1^{1/2} + \gamma_2^{1/2}}} \, \dd x + \int_{\R^3} \inner{a w^0 + b w^3}{\varphi^0} \, \dd x \\
- \int_{\R^3} \inner{d (\mu_1^{1/2} + \mu_2^{1/2})}{d \frac{\inner{w^3}{\varphi^3}}{\mu_1^{1/2} + \mu_2^{1/2}}} \, \dd x + \int_{\R^3} \inner{d w^0 + c w^3}{\varphi^3} \, \dd x
\end{gather*}
for any $ \varphi = \varphi^0 + \varphi^3 $ with $ \varphi^l \in H^1(\R^3 ; \Lambda^l \R^3) $. Then $ w \in \dot{X}^{1/2}_\zeta $ and \eqref{eq:final-u-v} reads
\begin{equation}
\Delta_\zeta w + Q w = 0. \label{eq:Azeta+Q}
\end{equation}
Here we have identified $ \zeta $ with a $ 1 $-form also denoted by $ \zeta $. Following the same argument as in Lemma \ref{lem:CGO-HL+Q1}, we can prove
\begin{equation}
\norm{Q}{}{\dot{X}^{1/2}_\zeta \rightarrow \dot{X}^{-1/2}_\zeta} = o(\textbf{1}(|\zeta|)) \label{es:Quniquecont}
\end{equation}
as $ |\zeta| $ becomes large. Then, Remark \ref{rem:uniqueness}, identity \eqref{id:Dz-k2=1}, \eqref{es:Quniquecont} and the Banach fixed-point theorem imply that \eqref{eq:Azeta+Q} has a unique solution belonging to $ \dot{X}^{1/2}_\zeta $. Therefore, $ w = 0 $ which in turn implies $ f = g = 0 $.
\end{proof}

\appendix
\section{The framework of differential forms}\label{apx:A}
Since the tools used in this paper are scattered in the literatures, to make the paper more self-contained, we summarized them in this Appendix. We start with collecting several basics required in the framework of differential forms (see \cite{Ta} and \cite{Fe} for some details of differential forms and Grassman graded alegra),
and the basic functional spaces and properties for the current discussion of PDEs. Then we show a useful identity used in the paper, and end our discussion with recalling basic facts about the Fourier transform of graded forms.

\subsection{Tools of multivariable calculus}
For $ x \in \R^n $ and $ n \in \N\backslash\{0\}$, let $ T_x \R^n $ denote the complex vector space of distributions $ X $ of order one in $ \R^n $ satisfying $ \supp X = \{ x \} $ and $ \duality{X}{c} = 0 $ for any constant function $c$ (See Theorem 2.3.4 in \cite{H} for the justification of this definition). Such $ X $ can be uniquely extended to a linear form on $C^1(\R^n)$, the space of continuously differentiable functions in $ \R^n $. Let $ \partial_{x^j}|_x $ denote the distribution given by
\[ \Duality{\partial_{x^j}|_x}{\phi} = \partial_{x^j} \phi (x) \]
for any $ \phi \in C^1(\R^n) $. The set $ \{ \partial_{x^1}|_x, \dots, \partial_{x^n}|_x \} $ is a base of $ T_x \R^n $. Let $ T_x^\ast \R^n $ denote the dual vector space of $ T_x \R^n $ with $ \{ dx^1|_x, \dots, dx^n|_x \} $ being the dual base. We define on $ T_x^\ast \R^n $ the inner product $ \inner{\centerdot}{\centerdot} $ given by the bilinear extension of $ \inner{dx^j|_x}{dx^k|_x} = \delta_{jk} $, with $ \delta_{\centerdot \centerdot} $ being the Kronecker delta. Note that it is not an Hermitian product.

\subsubsection{Differential forms}
Let $ \Lambda^l \R^n $ with $ l \in \{ 0, 1, \dots, n \} $ and $ n \geq 2 $ denote the smooth complex vector bundle over $ \R^n $ whose fiber at $ x \in \R^n $ consists in $ \Lambda^l T_x^\ast \R^n $ the $ l $-fold exterior product of $ T_x^\ast \R^n $. By convention, a $ 0 $-fold is just a complex number and a $ 1 $-fold is an element of $ T_x^\ast \R^n $. Let $ E $ be a non-empty subset of $ \R^n $, an $ l $-form on $ E $ is a section $ u $ of $ \Lambda^l \R^n $ over $ E $, so $ u(x) = u|_x \in \Lambda^l T_x^\ast \R^n $ for any $ x \in E $. Any $ l $-form on $ E $ with $ l \in \{ 1, \dots, n \} $ can be written as
\[ u = \sum_{\alpha \in S^l} u_\alpha\, dx^{\alpha_1} \wedge \dots \wedge dx^{\alpha_l} \]
with $ S^l = \{ (\alpha_1, \dots, \alpha_l) \in \{ 1, \dots, n \}^l : \alpha_1 < \dots < \alpha_l \} $ and $ u_\alpha : E \longrightarrow \C $. It is convenient to call $ u_\alpha $ with $\alpha \in S^l $ the component functions of $ u $.

The exterior product of an $ l $-form $ u $ and an $ m $-form $ v $, both on $ E $, is denoted by $ (u \wedge v) (x) = u|_x \wedge v|_x $ for any $ x \in E $. Recall that the exterior product is bilinear, associative and anti-commutative:
\begin{equation}
u \wedge v = (-1)^{lm} v \wedge u. \label{id:anti-commutation}
\end{equation}
Since a $ 0 $-form $ v $ on $ E $ is nothing but a map from $ E $ to $ \C $, it holds that $ u \wedge v = v \wedge u = vu $ for any $ l $-form $ u $ on $ E $.

The inner product of two $ l $-forms on $ E $ with $ l \in \{ 2, \dots, n \} $ can be defined at each point $ x \in E $ as the bilinear extension of
\[ \inner{(dx^{\alpha_1} \wedge \dots \wedge dx^{\alpha_l})|_x}{(dx^{\beta_1} \wedge \dots \wedge dx^{\beta_l})|_x} = \det \inner{dx^{\alpha_j}|_x}{dx^{\beta_k}|_x}, \]
where the right hand side stands for the determinant of the matrix \[ \left( \inner{dx^{\alpha_j}|_x}{dx^{\beta_k}|_x} \right)_{jk}. \]
The inner product of two $ 0 $-forms is just the usual product of functions.
The inner product on $ l $-forms can be immediately extended to graded forms $ u(x) = \sum_0^n u^l(x) $ and $ v(x) = \sum_0^n v^l(x) $ on $ E $, with $ u^l $ and $ v^l $ $ l $-forms on $ E $, as follows
\[ \inner{u}{v}(x) = \sum_{l = 0}^n \inner{u^l|_x}{v^l|_x}. \]
Associated to this inner product, we consider the norm satisfying $ |u|^2 = \inner{u}{\overline{u}} $.

Let now $ T_x^\ast \R^n $ be endowed with an orientation. The Hodge star operator of an $ l $-form on $ E $ with $ l \in \{ 1, \dots, n - 1 \} $ is defined at each point $ x \in E $ as the linear extension of
\[ \ast (dx^{\alpha_1} \wedge \dots \wedge dx^{\alpha_l})|_x = (dx^{\beta_1} \wedge \dots \wedge dx^{\beta_{n - l}})|_x, \]
where $ (\beta_1, \dots, \beta_{n - l}) \in \{ 1, \dots , n \}^{n - l} $ is chosen such that
\[ \{ dx^{\alpha_1}, \dots dx^{\alpha_l}, dx^{\beta_1}, \dots, dx^{\beta_{n - l}} \} \]
is a positive base of $ T_x^\ast \R^n $. The case of $ 0 $-forms and $ n $-forms follows from
\[ \ast 1|_x = (dx^1 \wedge \dots \wedge dx^n)|_x, \qquad \ast (dx^1 \wedge \dots \wedge dx^n)|_x = 1|_x, \]
where $ 1 $ denotes the constant function taking the value $ 1 $ at any point. Now, if $ u $ and $ v $ are $ l $-forms on $ E $, then
\begin{align}
& \ast \ast u(x) = (- 1)^{l (n - l)} u(x) \label{eq:astast}, \\
& \inner{u}{v}(x) = \ast (u|_x \wedge \ast v|_x) = \ast (v|_x \wedge \ast u|_x), \label{eq:inner_astwedgeast} \\
& \inner{u}{v} = \inner{\ast u}{ \ast v}. \label{eq:astinner}
\end{align}
Let $ u $ be an $ l $-form on $ E $ and let $ v $ an $ m $-form on $ E $. The vee product of $ v $ and $ u $ at each point $ x \in E $ is defined as
\begin{equation}
(v \vee u)(x) = (-1)^{(n + m - l)(l - m)} \ast (v|_x \wedge \ast u|_x). \label{def:veePRODUCT}
\end{equation}
Note that whenever $ m > l $, $ (v \vee u)(x) = 0 $ for all $ x \in E $. The vee product is bilinear but it is neither associative nor commutative. The product satisfies
\begin{equation}
\inner{w \wedge v}{u} = \inner{w}{v \vee u} \label{eq:veeWEDGE}
\end{equation}
for any $ k $-form $ w $ on $ E $.

\begin{proposition} \sl
If $ u $ and $ v $ are $ 1 $-forms and $ w $ is an $ l $-form with $ l \in \{ 0, \dots, n \} $, then
\begin{equation}
u \vee (v \wedge w) - v \wedge (u \vee w) = (-1)^l \inner{u}{v} w. \label{eq:1v1wl}
\end{equation}
\end{proposition}

\begin{corollary} \label{cor:innerinner} \sl
If $ u^1 $ and $ v^1 $ are $ 1 $-forms and $ u^l $ and $ v^l $ are $ l $-forms with $ l \in \{ 0, \dots, n \} $, then
\[ \inner{u^1 \vee u^l}{v^1 \vee v^l} + \inner{v^1 \wedge u^l}{u^1 \wedge v^l} = \inner{u^1}{v^1} \inner{u^l}{v^l}. \]
\end{corollary}
\begin{proof}
Since
\begin{gather*}
\inner{u^1 \vee u^l}{v^1 \vee v^l} + \inner{v^1 \wedge u^l}{u^1 \wedge v^l} = \\
= (-1)^l \inner{u^1 \vee (v^1 \wedge u^l) - v^1 \wedge (u^1 \vee u^l)}{v^l}
\end{gather*}
the identity follows from \eqref{eq:1v1wl}.
\end{proof}

Let $ G $ be a non-empty open subset of $ \R^n $ and $ k $ a positive integer, an $ l $-form $ u $ on $ G $ with $ l \in \{ 1, \dots, n \} $ is said to be $ k $-times continuously differentiable if its component functions are $ k $-times continuously differentiable in $ G $, we write $ u \in C^k (G;\Lambda^l \R^n) $. If $ u \in C^k (G;\Lambda^l \R^n) $ for any positive integer $ k $, we say that $ u $ is smooth and we write $ u \in C^\infty (G;\Lambda^l \R^n) $. Furthermore, $ u \in C^k (G;\Lambda^l \R^n) $ (resp., $ u \in C^\infty (G;\Lambda^l \R^n) $) is said to be compactly supported if its component functions are compactly supported in $ G $, in which case we write $ u \in C^\infty_0 (G;\Lambda^l \R^n) $ (resp., $ u \in C^\infty_0 (G;\Lambda^l \R^n) $). These definitions are naturally generalized to $0$-forms, where the conventional function space notations are also used.

The exterior derivative of $ u\in C^1(G;\Lambda^0\R^n) $ is a $ 1 $-form defined by
\[ du|_x (X) = \Duality{X}{\chi_x u} \]
for each $ x \in G $ and $ X \in T_x \R^n $. Here $ \chi_x \in C^\infty_0 (G) $ with $ \chi_x(x) = 1 $ on $G$, and $ \chi_x u $ is understood as the extension of $u$ by zero outside $ G $. 
The exterior derivative of $ u\in C^1(G;\Lambda^l\R^n) $ with $ l \in \{ 1, \dots, n \} $ is defined by
\[ du = \sum_{\alpha \in S^l} du_\alpha \wedge dx^{\alpha_1} \wedge \dots \wedge dx^{\alpha_l}. \]
Recall that $ d(du) = 0 $ for any $u\in C^2(G;\Lambda^l\R^3)$ and
\begin{equation}
d (u \wedge v) = du \wedge v + (-1)^l u \wedge dv, \label{for:dWEDGE}
\end{equation}
for any $u\in C^1(G;\Lambda^l\R^3)$ and $v\in C^1(G;\Lambda^m\R^3)$.
\subsubsection{Symmetric tensors}
Let $ \Sigma^l \R^n $ with $ l \in \N $ and $ n \geq 2 $ denote the smooth complex vector bundle over $ \R^n $ whose fiber at $ x \in \R^n $ consists in $ \Sigma^l T_x^\ast \R^n $ the $ l $-fold symmetric tensor product of $ T_x^\ast \R^n $. By convention, a $ 0 $-fold is just a complex number and a $ 1 $-fold is an element of $ T_x^\ast \R^n $. Let $ E $ be a non-empty subset of $ \R^n $, an $ l $-symmetric tensor on $ E $ is a section $ u $ of $ \Sigma^l \R^n $ over $ E $, so $ u(x) = u|_x \in \Sigma^l T_x^\ast \R^n $ for any $ x \in E $. Any $ l $-symmetric tensor on $ E $ with $ l \in \{ 1, \dots, n \} $ can be written as
\[ u = \sum_{\alpha \in T^l} u_\alpha\, dx^{\alpha_1} \odot \dots \odot dx^{\alpha_l} \]
with $ T^l = \{ (\alpha_1, \dots, \alpha_l) \in \{ 1, \dots, n \}^l : \alpha_1 \leq \dots \leq \alpha_l \} $ and $ u_\alpha : E \longrightarrow \C $. It is convenient to call $ u_\alpha $ with $\alpha \in T^l $ the component functions of $ u $ and to point out that $ \Sigma^l T_x^\ast \R^n = \Lambda^l T_x^\ast \R^n $ for $ l \in \{ 0, 1 \} $, which in turn implies $ \Sigma^l \R^n = \Lambda^l \R^n $ for $ l \in \{ 0, 1 \} $.

The symmetric tensor product of an $ l $-symmetric tensor $ u $ and an $ m $-symmetric tensor $ v $, both on $ E $, is denoted by $ (u \odot v) (x) = u|_x \odot v|_x $ for any $ x \in E $. Recall that the symmetric tensor product is bilinear, associative and commutative. Moreover, if $ u $ and $ v $ are $ 1 $-symmetric tensors, then
\[ u \odot v = \frac{1}{2} (u \otimes v + v \otimes u). \]

The inner product of two $ l $-symmetric tensors on $ E $ with $ l \in \N \setminus \{ 0, 1 \} $ can be defined at each point $ x \in E $ as the bilinear extension of
\[ \inner{(dx^{\alpha_1} \odot \dots \odot dx^{\alpha_l})|_x}{(dx^{\beta_1} \odot \dots \odot dx^{\beta_l})|_x} = \left| \det \inner{dx^{\alpha_j}|_x}{dx^{\beta_k}|_x} \right|. \]

Let $ G $ be a non-empty open subset of $ \R^n $, an $ l $-symmetric tensor $ u $ on $ G $ with $ l \in \N $ is said to be $ k $-times continuously differentiable if its component functions are $ k $-times continuously differentiable in $ G $, and we write $ u \in C^k (G;\Sigma^l \R^n) $. Furthermore, $ u \in C^k (G;\Sigma^l \R^n) $ with $ l \in \N $ is said to be compactly supported if its component functions are compactly supported in $ G $, and we write $ u \in C^k_0 (G;\Sigma^l \R^n) $. These definitions extend naturally to $ 0 $-symmetric tensors on $ G $.

The symmetric derivative of a smooth $ l $-symmetric tensor $ u $ on $ G $ with $ l \in \N \setminus \{ 0 \} $ is defined by
\[ Du = \sum_{\alpha \in T^l} du_\alpha \odot dx^{\alpha_1} \odot \dots \odot dx^{\alpha_l}. \]
\subsection{Functional spaces}
Let $ L^1_\mathrm{loc} (E; \Lambda^l \R^n) $ denote the space of locally integrable $ l $-forms (whose component functions are in $L^1_\mathrm{loc}(E)$) modulo those which vanish almost everywhere (a. e. for short) in $ E $. The space $ L^p (E; \Lambda^l \R^n) $, with $ p \in [1, +\infty) $, consists of all $ u \in L^1_\mathrm{loc} (E; \Lambda^l \R^n) $ such that
\[ \int_E \inner{u}{\overline{u}}^{p/2} \, \dd x < +\infty. \]
Endowed with norm
\[ \norm{u}{}{L^p (E; \Lambda^l \R^n)} = \left( \int_E \inner{u}{\overline{u}}^{p/2} \, \dd x \right)^{1/p}, \]
$ L^p (E; \Lambda^l \R^n) $ is a Banach space. Moreover, $ L^2 (E; \Lambda^l \R^n) $ is a Hilbert space.

Let $ u \in L^1_\mathrm{loc} (G; \Lambda^l \R^n) $ with $ l \in \{ 1, \dots, n \} $. We say that $ v \in L^1_\mathrm{loc} (G; \Lambda^{l - 1} \R^n) $ is the formal adjoint derivative of $ u $, denoted by $ v = \delta u $, if
\[ \int_{G} \inner{v}{w} \, \dd x = \int_{G} \inner{u}{d w} \, \dd x \]
for any $ w \in C^1_0 (G; \Lambda^{l - 1} \R^n) $. If $ u \in L^1_\mathrm{loc} (G; \Lambda^0 \R^n) $, we define $ \delta u = 0 $. For all $ u\in L^1_\mathrm{loc} (G; \Lambda^l \R^n) $ with $ l \in \{ 0, \dots, n \} $ such that $ \delta u \in L^1_\mathrm{loc} (G; \Lambda^{l - 1} \R^n) $ one has $ \delta (\delta u) = 0 $. Moreover, if $ u \in C^1 (G;\Lambda^l \R^n) $, then
\begin{equation}
\delta u = (-1)^{n(l + 1) + 1} \ast d \ast u. \label{eq:DELTAd}
\end{equation}

\begin{proposition} \sl
Consider $ u \in L^1_\mathrm{loc} (G;\Lambda^l \R^n) $ and $ v \in C^1 (G;\Lambda^m \R^n) $. If $ \delta u \in L^1_\mathrm{loc} (G;\Lambda^{l - 1} \R^n) $, then $ \delta (v \vee u) \in L^1_\mathrm{loc} (G;\Lambda^{l - m - 1} \R^n) $ and
\begin{equation}
\delta (v \vee u) = (- 1)^{l - m} dv \vee u + v \vee \delta u. \label{for:DELTAvee}
\end{equation}
\end{proposition}

Let $ u \in L^1_\mathrm{loc} (G; \Lambda^l \R^n) $ with $ l \in \{ 0, \dots, (n - 1) \} $. We say that $ v \in L^1_\mathrm{loc} (G; \Lambda^{l + 1} \R^n) $ is the (weak) exterior derivative of $ u $, denoted by $ v = d u $ if
\[ \int_{G} \inner{v}{w} \, \dd x = \int_{G} \inner{u}{\delta w} \, \dd x \]
for any $ w \in C^1_0 (G; \Lambda^{l + 1} \R^n) $. If $ u \in L^1_\mathrm{loc} (G; \Lambda^n \R^n) $, we define $ d u = 0 $. For all $ u \in L^1_\mathrm{loc} (G; \Lambda^l \R^n) $ with $ l \in \{ 0, \dots, n \} $ such that $ d u \in L^1_\mathrm{loc} (G; \Lambda^{l + 1} \R^n) $ one has $ d (d u) = 0 $.

\begin{proposition} \sl
Let $ u \in L^1_\mathrm{loc} (G; \Lambda^l \R^n) $ such that $ \delta u \in L^1_\mathrm{loc} (G; \Lambda^{l - 1} \R^n) $, then $ \ast d \ast u \in L^1_\mathrm{loc} (G; \Lambda^{l - 1} \R^n) $ and
\begin{equation}
\delta u = (-1)^{n(l + 1) + 1} \ast d \ast u. \label{eq:weakDELTAd}
\end{equation}
\end{proposition}

We now present certain Sobolev spaces of forms, in which our PDEs are discussed. Let $ H^d (G; \Lambda^l \R^n) $ (resp., $H^\delta(G;\Lambda^l\R^n)$) denote the space of $ u \in L^2 (G; \Lambda^l \R^n) $ such that $ du \in L^2 (G; \Lambda^{l + 1} \R^n) $ (resp. $\delta u\in L^2(G; \Lambda^{l-1}\R^n)$), endowed with the norm
\[ \norm{u}{}{H^d (G; \Lambda^l \R^n)} = \left( \norm{u}{2}{L^2 (G; \Lambda^l \R^n)} + \norm{du}{2}{L^2 (G; \Lambda^{l + 1} \R^n)} \right)^{1/2} \]
\[\left(\mbox{resp., }\; \norm{u}{}{H^\delta (G; \Lambda^l \R^n)} = \left( \norm{u}{2}{L^2 (G; \Lambda^l \R^n)} + \norm{\delta u}{2}{L^2 (G; \Lambda^{l - 1} \R^n)} \right)^{1/2}\right). \]
It is observed that $  H^d (G; \Lambda^l \R^n)  $ (resp., $H^\delta(G;\Lambda^l\R^n)$) is a Hilbert space and $ C^1_0 (\R^n; \Lambda^l \R^n) $ is dense in it. Let $ H^d_\mathrm{loc} (\R^n; \Lambda^l \R^n) $ (resp., $ H^\delta_\mathrm{loc} (\R^n; \Lambda^l \R^n) $)  denote the space of $ u \in L^1_\mathrm{loc} (\R^n; \Lambda^l \R^n) $ such that $ u|_U \in H^d (U; \Lambda^l \R^n) $ (resp., $ u|_U \in H^\delta (U; \Lambda^l \R^n) $) for any bounded non-empty open subset $ U $ in $ \R^n $.


Finally, by a density argument it holds that
\begin{equation}
\int_{\R^n} \inner{d u}{v} \, \dd x = \int_{\R^n} \inner{u}{\delta v} \, \dd x \label{eq:weakd-delta}
\end{equation}
for all $ u \in H^d (\R^n; \Lambda^{l - 1} \R^n) $ and $ v \in H^\delta (\R^n; \Lambda^l \R^n) $ with $ l \in \{ 1, \dots, n \} $.

\subsection{Traces}\footnote{For more datails on traces see \cite{Mi} and \cite{Sc}.} \label{app:traces}
Let $ U $ be a non-empty bounded open subset of $ \R^n $. Let $ H^1 (U; \Lambda^l \R^n) $ denote the space of all $ u \in L^2 (U; \Lambda^l \R^n) $ whose component functions $ u_\alpha $ satisfy $ d u_\alpha \in L^2 (U; \Lambda^1 \R^n) $ for all $ \alpha \in S^l $, endowed with the norm
\begin{equation}
\norm{u}{}{H^1 (U; \Lambda^l \R^n)} = \left( \norm{u}{2}{L^2(U; \Lambda^l \R^n)} + \sum_{\alpha \in S^l} \norm{d u_\alpha}{2}{L^2(U; \Lambda^1 \R^n)} \right)^{1/2}. \label{def:H1norm}
\end{equation}
Given $ G $ a non-empty open subset of $ \R^n $, by $ H^1_\mathrm{loc} (G; \Lambda^l \R^n) $ we denote the space of $ u \in L^1_\mathrm{loc} (G; \Lambda^l \R^n) $ such that $ u|_U \in H^1 (U; \Lambda^l \R^n) $ for any bounded non-empty open subset $ U $ of $ G $.


It is a consequence of \eqref{eq:weakDELTAd} that, for any $ u \in H^1 (U; \Lambda^l \R^n) $, one has
\begin{equation}
\norm{u}{}{H^\delta (U; \Lambda^l \R^n)} \leq \norm{u}{}{H^1 (U; \Lambda^l \R^n)}. \label{es:HdeltaH1}
\end{equation}
Let $ H^1_0 (U; \Lambda^l \R^n) $ denote the closure in $ H^1 (U; \Lambda^l \R^n) $ of $ C^\infty_0 (U; \Lambda^l \R^n) $ modulo those vanishing a. e. in $ U $. We then define the space 
\[ TH^1 (\partial U; \Lambda^l \R^n) = H^1 (U; \Lambda^l \R^n) / H^1_0 (U; \Lambda^l \R^n).\] 
If $ f \in TH^1 (\partial U; \Lambda^l \R^n) $, let $ u_f \in H^1 (U; \Lambda^l \R^n) $ denote a representative of $ f $. This space can be endowed with the norm
\[ \norm{f}{}{TH^1 (\partial U; \Lambda^l \R^n)} = \inf \{ \norm{u}{}{H^1 (U; \Lambda^l \R^n)} : u - u_f \in H^1_0 (U; \Lambda^l \R^n) \}. \]
Let $ TH^1 (\partial U; \Lambda^l \R^n)^\ast $ denote the dual space of $ TH^1 (\partial U; \Lambda^l \R^n) $ with the functional $ \norm{\centerdot}{}{TH^1 (\partial U; \Lambda^l \R^n)^\ast} $ standing for the dual norm.

The latter spaces will be used as auxiliary spaces to define certain traces on $ H^d (U; \Lambda^l \R^n) $ and $ H^\delta (U; \Lambda^l \R^n) $. Firstly, define the $ d $-trace of $ v \in H^d (U; \Lambda^l \R^n) $ with $ l \in \{ 0, \dots, n - 1 \} $ as
\[ \Duality{d\tr\, v}{f} = \int_U \inner{dv}{u} \, \dd x - \int_U \inner{v}{\delta u} \, \dd x \]
for any $ f \in TH^1 (\partial U; \Lambda^{l + 1} \R^n) $ where $ u \in H^1 (U; \Lambda^{l+1} \R^n) $ such that $ u - u_f \in H^1_0 (U; \Lambda^{l+1} \R^n) $. Since \eqref{es:HdeltaH1} holds, we have
\[ \Duality{d\tr\, v}{f} \leq \norm{v}{}{H^d (U; \Lambda^l \R^n)} \norm{u}{}{H^1 (U; \Lambda^{l+1} \R^n)} \]
for all $ u \in H^1 (U; \Lambda^{l+1} \R^n) $ such that $ u - u_f \in H^1_0 (U; \Lambda^{l+1} \R^n) $. Hence $ d\tr\, v \in TH^1 (\partial U; \Lambda^{l + 1} \R^n)^\ast $ and
\[ \norm{d\tr\, v}{}{TH^1 (\partial U; \Lambda^{l + 1} \R^n)^\ast} \leq \norm{v}{}{H^d (U; \Lambda^l \R^n)}. \]
This motivates the definition of $ TH^d(\partial U; \Lambda^{l + 1} \R^n) $ to be the space of all $ g \in TH^1 (\partial U; \Lambda^{l + 1} \R^n)^\ast $ such that $ d\tr\, v = g $ for some $ v \in H^d (U; \Lambda^l \R^n) $. The endowed norm is then given by
\[ \norm{g}{}{TH^d(\partial U; \Lambda^{l + 1} \R^n)} = \inf \{ \norm{v}{}{H^d (U; \Lambda^l \R^n)} : d\tr\, v = g \}. \]
Finally, we define the $ \delta $-trace of $ v \in H^\delta (U; \Lambda^l \R^n) $ with $ l \in \{ 1, \dots, n \} $ as
\[ \Duality{\delta \tr\, v}{f} = (-1)^l \int_U \inner{\delta v}{u} \, \dd x - (-1)^l \int_U \inner{v}{d u} \, \dd x \]
for any $ f \in TH^1 (\partial U; \Lambda^{l - 1} \R^n) $ where $ u \in H^1 (U; \Lambda^{l - 1} \R^n) $ such that $ u - u_f \in H^1_0 (U; \Lambda^{l - 1} \R^n) $. 
Similarly we would have $ \delta \tr\, v \in TH^1 (\partial U; \Lambda^{l - 1} \R^n)^\ast $ and
\[ \norm{\delta \tr\, v}{}{TH^1 (\partial U; \Lambda^{l - 1} \R^n)^\ast} \leq \norm{v}{}{H^\delta (U; \Lambda^l \R^n)}. \]
Moreover, we define $ TH^\delta (\partial U; \Lambda^{l - 1} \R^n) $ the space consisting of all $ g $ belonging to $ TH^1 (\partial U; \Lambda^{l - 1} \R^n)^\ast $ such that there exists $ v \in H^\delta (U; \Lambda^l \R^n) $ with $ \delta \tr\, v = g $ with norm
\[ \norm{g}{}{TH^\delta (\partial U; \Lambda^{l - 1} \R^n)} = \inf \{ \norm{v}{}{H^\delta (U; \Lambda^l \R^n)} : \delta \tr\, v = g \}. \] Then we will need the following lemma about these spaces
\begin{lemma} \label{lem:traces} \sl Given the definitions above,
\begin{itemize}
\item[(a)] if $ u \in H^d (U; \Lambda^l \R^n) $ with $ l \in \{ 0, \dots, n - 2 \} $ and $ d \tr u = 0 $, then $ d \tr (d u) = 0 $;
\item[(b)] if $ u \in H^\delta (U; \Lambda^l \R^n) $ with $ l \in \{ 2, \dots, n \} $ and $ \delta \tr u = 0 $, then $ \delta \tr (\delta u) = 0 $.
\end{itemize}
\end{lemma}
\begin{proof} In order to prove (a), let us consider the bounded linear operator $ (\nu \vee \centerdot) : TH^d(\partial U; \Lambda^l \R^n) \longrightarrow TH^\delta (\partial U; \Lambda^{l - 1} \R^n)^\ast $ given by
\[ \Duality{\nu \vee f}{g} = \int_U \inner{du}{v} \, \dd x - \int_U \inner{u}{\delta v} \, \dd x \]
where $ u \in H^d(U; \Lambda^{l - 1} \R^n) $, $ v \in H^\delta (U; \Lambda^l \R^n) $, $ d \tr u = f $ and $ \delta \tr v = g $. Here $ TH^\delta (\partial U; \Lambda^{l - 1} \R^n)^\ast $ denotes the dual of $ TH^\delta (\partial U; \Lambda^{l - 1} \R^n) $. Let $ u $ be as in (a) and $ g \in TH^1(U; \Lambda^{l + 2} \R^n) $, then
\[ \Duality{d \tr (d u)}{g} = - \Duality{\nu \vee d \tr u}{\delta \tr (\delta v_g)}, \]
where $ v_g \in H^1(U; \Lambda^{l + 2} \R^n) $ denotes a representative of $ g $. Therefore (a) holds.

Similar proof applies to (b) by considering the operator $ (\nu \vee \centerdot) : TH^\delta(\partial U; \Lambda^l \R^n) \longrightarrow TH^d (\partial U; \Lambda^{l + 1} \R^n)^\ast $ defined by
\[ \Duality{\nu \wedge f}{g} = (-1)^{l + 1} \int_U \inner{\delta u}{v} \, \dd x - (-1)^{l + 1} \int_U \inner{u}{d v} \, \dd x \]
where $ u \in H^\delta (U; \Lambda^{l + 1} \R^n) $, $ v \in H^d (U; \Lambda^l \R^n) $, $ \delta \tr u = f $ and $ d \tr v = g $. We will leave the proof to the readers.

\end{proof}

\subsection{A useful identity}
Given $ G$ a non-empty open subset of $ \R^n $, let $ L^1_\mathrm{loc} (G; \Sigma^l \R^n) $ denote the space of locally integrable $ l $-symmetric tensors (whose component functions are in $L^1_\mathrm{loc}(E)$) modulo those which vanish a. e. in $ E $.

For $u\in L^1_\mathrm{loc} (G; \Sigma^l \R^n) $ with $ l \in \N \setminus \{ 1, 2 \} $, we say that $ v \in L^1_\mathrm{loc} (G; \Sigma^{l - 1} \R^n) $ is the formal adjoint (symmetric) derivative of $ u $, denoted by $ v = D^\ast u $, if
\[ \int_{G} \inner{v}{w} \, \dd x = \int_{G} \inner{u}{D w} \, \dd x \]
for any $ w \in C^1_0 (G; \Sigma^{l - 1} \R^n) $.

Note that if $ u = \sum_{j = 1}^n u_j dx^j $ and $ v = \sum_{j = 1}^n v_j dx^j $ such that $ u \odot v \in L^1_\mathrm{loc} (G; \Sigma^2 \R^n) $ and $ D^\ast (u \odot v) \in L^1_\mathrm{loc} (G; \Sigma^1 \R^n) $, then
\begin{equation}
D^\ast (u \odot v) = - \sum_{k = 1}^n \left( \sum_{j = 1}^n \partial_{x^j} (u_j v_ k + u_k v_j) \right) dx^k. \label{eq:DsymCOORDINATES}
\end{equation}

\begin{proposition} \label{prop:vdwdelta} \sl
Given $ u $ and $ v $ in $ H^1_\mathrm{loc} (G; \Lambda^1 \R^n) $, we have $ d \inner{u}{v} $ and $ D^\ast (u \odot v) $ belong to $ L^1_\mathrm{loc} (G; \Lambda^1 \R^n) $ and the following identity holds
\begin{equation*}
u \vee d v + v \vee d u + \delta u \vee v + \delta v \vee u = d \inner{u}{v} + D^\ast (u \odot v).
\end{equation*}
\end{proposition}

\subsection{Local regularity}
Here we prove a local regularity lemma for the operator $ (d + \delta) \sum_0^n (-1)^l $.
\begin{lemma} \label{lem:H^delta-H^d} \sl Let $ v = \sum_0^n v^l $ be such that $ v^l \in L^2_\mathrm{loc} (\R^n; \Lambda^l \R^n) $ and
\[ (d + \delta) \sum_{l = 0}^n (-1)^l v^l \in \bigoplus_{l = 0}^n L^2_\mathrm{loc} (\R^n; \Lambda^l \R^n). \]
Then $ v^l \in H^1_\mathrm{loc} (\R^n; \Lambda^l \R^n) $ for $ l \in \{ 0, \dots, n \} $.
\end{lemma}
\begin{proof} By using Corollary \ref{cor:innerinner} and the following identity
\[ \inner{\xi \wedge \widehat{\phi^{l - 1}}(\xi)}{\overline{\xi \vee \widehat{\phi^{l + 1}}(\xi)}} = 0, \]
we can check that
\begin{equation}
  \label{id:apriori}
  \norm{\phi}{2}{L^2} = \norm{\phi}{2}{H^{-1}} + \norm{(d + \delta) \sum_{l = 0}^n (-1)^l \phi^l }{2}{H^{-1}}
\end{equation}
for all $ \phi = \sum_0^n \phi^l $ such that $ \phi^l \in L^2 (\R^n; \Lambda^l \R^n) $. 
Here we are using the notation
$ \norm{\varphi}{2}{Y} = \sum_0^n \norm{\varphi^l}{2}{Y(\R^n; \Lambda^l \R^n)} $
for $ \varphi = \sum_0^n \varphi^l $ with $ \varphi^l \in Y(\R^n; \Lambda^l \R^n) $, where $ Y $ denotes either $ L^2 $ or $ H^{-1} $. Remind that $\|\varphi^l\|^2_{H^{-1}(\R^n; \Lambda^l \R^n)} = \int_{\R^n} (1+|\xi|^2)^{-1} |\widehat{\varphi^l}(\xi)|^2\, \dd \xi$.

Let $ \psi $ be a compactly supported smooth function in $ \R^n $ and let $ \Delta^j_h \phi $ be defined as
\begin{equation*}
  \Delta^j_h \phi(x) = \frac{1}{h} ( \phi(x + he_j) - \phi(x) )
\end{equation*}
with $ \phi $ as in \eqref{id:apriori}, $ h $ a positive parameter and $ e_j $ the $ j $-th element of the orthonormal basis of $ \R^n $. By \eqref{id:apriori} and the commutativity between $ \Delta^j_h $ and $ (d + \delta) \sum_0^n (-1)^l $, we have
\begin{equation}
  \label{id:aposteriori}
  \norm{\Delta^j_h (\psi v)}{2}{L^2} = \norm{\Delta^j_h (\psi v)}{2}{H^{-1}} + \norm{\Delta^j_h (d + \delta) \sum_{l = 0}^n (-1)^l (\psi v^l) }{2}{H^{-1}}.
\end{equation}
Since
\begin{equation*}
  (d + \delta) \sum_{l = 0}^n (-1)^l (\psi v^l) = \psi (d + \delta) \sum_{l = 0}^n (-1)^l v^l + \sum_{l = 0}^n (-1)^l d \psi \wedge v^l + d \psi \vee v^l
\end{equation*}
and $ v $ and $  (d + \delta) \sum_0^n (-1)^l v^l $ belong to $ \bigoplus_0^n L^2_\mathrm{loc} (\R^n; \Lambda^l \R^n) $, the statement of the result follows by making the parameter $ h $ goes to zero in the identity \eqref{id:aposteriori}\footnote{See more details in Theorem (6.19) of \cite{Fo}.}.
\end{proof}

\subsection{Fourier transform of forms and operator $\Delta_\zeta$}\label{appdx:FT}
An $l$-form $u$ with $l\in\{0,\ldots,n\}$ is said to belong to the Schwartz space $\mathcal{S}(\R^n;\Lambda^l\R^n)$ if its component functions $u_\alpha$ ($\alpha\in S^l$) are in the Schwartz space $\mathcal{S}(\R^n)$. We can define the space $\mathcal{S}'(\R^n;\Lambda^l\R^n)$ of $l$-form-valued tempered distributions similarly. The Fourier Transform of $u\in\sh(\R^n;\Lambda^l\R^n)$ is then defined by
\[\widehat{u}=\sum_{\alpha\in S^l}\widehat{u_\alpha} d\xi^{\alpha_1}\wedge\ldots\wedge d\xi^{\alpha_l}\in \mathcal{S}(\R^n;\Lambda^l\R^n).\]
The Fourier Transform $\widehat{u}$ for $u\in\sh'(\R^n;\Lambda^l\R^n)$ can be defined by duality. One can easily verify the following identities for $u\in \mathcal{S}(\R^n;\Lambda^l\R^n)$
\begin{equation}\label{eq:sym-d-delta}
\widehat{du}(\xi)=i\xi\wedge \widehat{u}(\xi),\qquad \widehat{\delta u}(\xi)=i(-1)^{l} \xi\vee\widehat{u}(\xi)\end{equation}
where $\xi\in\R^3\backslash\{0\}$ can be viewed as a 1-form. For $u, v\in L^2(\R^n;\Lambda^l\R^n)$, we have
\begin{equation}\label{eq:FTdual}\int_{\R^n}\langle u,\overline{v}\rangle~\dd x=\int_{\R^n}\langle \widehat{u},\overline{\widehat{v}}\rangle~\dd x,\end{equation}
making Fourier Transform a unitary map on $L^2(\R^n;\Lambda^l\R^n)$. 

Given $ \zeta = \sum_1^n \zeta_j dx^j $ a constant $ 1 $-differential form in $ \R^n $. Consider the conjugated Hodge-Laplacian operator $ \Delta_\zeta = e_{-\zeta} (d\delta+\delta d) \circ e_\zeta $ where $ e_\zeta (x) = e^{\zeta \cdot x} $ and $ \zeta \cdot x = \sum_{1}^n \zeta_j x^j $. When acting on an $l$-form $u\in H^d(\R^n;\Lambda^l\R^n)\cap H^\delta(\R^n;\Lambda^l\R^n)$, it reads
\begin{equation}\label{eq:DELTA-zeta}
\begin{split}
\Delta_\zeta u=&(d\delta+\delta d)u+(-1)^ld(\zeta\vee u)+\zeta\wedge \delta u\\
&+\delta(\zeta\wedge u)+(-1)^{l+1}\zeta\vee du-\inner{\zeta}{\zeta} u,
\end{split}\end{equation}
(understood in the weak sense).
Furthermore, it is easy to verity that the symbol of $\Delta_\zeta$ is $|\xi|^2-2i\inner{\zeta}{\xi}-\inner{\zeta}{\zeta}$ by \eqref{eq:sym-d-delta}.

\end{document}